\documentclass[11pt,a4paper]{amsart}
\usepackage{mathtools}
\usepackage{amssymb}
\usepackage{dsfont,amsfonts}
\usepackage[left=25mm,top=25mm,bottom=25mm,right=25mm]{geometry}
\usepackage[english]{babel} 
\usepackage[utf8]{inputenc} 
\usepackage{cmap}           
\usepackage{hyperxmp}       
\usepackage{enumerate}
\usepackage{mleftright}
\usepackage{eqnarray}
\usepackage{fancybox}
\usepackage[pdfdisplaydoctitle = true,
colorlinks = true,
urlcolor = blue,
citecolor = blue,
linkcolor = blue,
pdfstartview = FitH,
pdfpagemode = UseNone,
bookmarksnumbered = true,
unicode=true]{hyperref} 
\usepackage{xfrac}

\newtheorem{thm}{Theorem}[section]

\newtheorem{lem}[thm]{Lemma}

\theoremstyle{definition}

\newtheorem{rem}[thm]{Remark}
\newtheorem{exmp}[thm]{Example}

\newcommand{\CC}{\mathbb{C}}                
\newcommand{\RR}{\mathbb{R}}                
\newcommand{\QQ}{\mathbb{Q}}                

\newcommand{\Ela}{\mathbb{E}\mathrm{la}}    
\newcommand{\HH}{\mathbb{H}}                
\newcommand{\Sym}{\mathbb{S}}               

\newcommand{\SO}{\mathrm{SO}}               
\newcommand{\octa}{\mathbb{O}}              

\newcommand{\ee}{\pmb{e}}                   
\newcommand{\uu}{\pmb{u}}                   
\newcommand{\xx}{\pmb{x}}                   

\newcommand{\qq}{\mathrm{q}}                

\newcommand{\blambda}{\pmb{\lambda}}        
\newcommand{\bsigma}{\pmb{\sigma}}          
\newcommand{\ba}{\mathbf{a}}
\newcommand{\bb}{\mathbf{b}}
\newcommand{\bc}{\mathbf{c}}
\newcommand{\bd}{\mathbf{d}}
\newcommand{\be}{\mathbf{e}}
\newcommand{\Id}{\mathbf{1}}                
\newcommand{\bh}{\mathbf{h}}

\newcommand{\bv}{\mathbf{v}}
\newcommand{\bw}{\mathbf{w}}

\newcommand{\bE}{\mathbf{E}}                
\newcommand{\bF}{\mathbf{F}}                
\newcommand{\bH}{\mathbf{H}}                
\newcommand{\bI}{\mathbf{I}}                

\newcommand{\bK}{\mathbf{K}}                
\newcommand{\bL}{\mathbf{L}}                
\newcommand{\bA}{\mathbf{A}}                
\newcommand{\bS}{\mathbf{S}}                
\newcommand{\bT}{\mathbf{T}}                
\newcommand{\bP}{\mathbf{P}}                
\newcommand{\lc}{\pmb \varepsilon}          

\DeclareMathOperator{\tr}{tr}
\DeclareMathOperator{\Orb}{Orb}

\DeclareMathOperator{\grad}{grad}
\DeclareMathOperator{\3dots}{\raisebox{-0.25ex}{\vdots}}
\newcommand{\dd}{\mathrm{d}}

\newcommand{\norm}[1]{\lVert#1\rVert}       
\newcommand{\set}[1]{\left\{#1\right\}}     

\hypersetup{
	pdfauthor={}, 
	pdftitle={The distance to cubic symmetry class as a polynomial optimization problem}, 
	pdfsubject={MSC 2020: xxx (xxx xxx)}, 
	pdfkeywords={}, 
	pdflang=en, 
}
\begin{document}

\title[The distance to cubic symmetry class as a polynomial optimization problem]{The distance to cubic symmetry class \protect\\ as a polynomial optimization problem}%

\author{P. Azzi}
\address[Perla Azzi]{Université Paris-Saclay, CentraleSupélec, ENS Paris-Saclay, CNRS, Laboratoire de Mécanique Paris-Saclay, 91190, Gif-sur-Yvette, France.}
\address{Sorbonne Université, Institut de Mathématiques de Jussieu-Paris Rive Gauche, 4 place Jussieu, 75005, Paris, France}
\email[Perla Azzi]{perla.azzi@ens-paris-saclay.fr}

\author{R. Desmorat}
\address[Rodrigue Desmorat]{Université Paris-Saclay, CentraleSupélec, ENS Paris-Saclay, CNRS, Laboratoire de Mécanique Paris-Saclay, 91190, Gif-sur-Yvette, France.}
\email{rodrigue.desmorat@ens-paris-saclay.fr}

\author{B. Kolev}
\address[Boris Kolev]{Université Paris-Saclay, CentraleSupélec, ENS Paris-Saclay, CNRS, Laboratoire de Mécanique Paris-Saclay, 91190, Gif-sur-Yvette, France.}
\email{boris.kolev@math.cnrs.fr}

\author{F. Priziac}
\address[Fabien Priziac]{Univ Bretagne Sud, CNRS UMR 6205, LMBA, F-56000 Vannes, France}
\email{fabien.priziac@univ-ubs.fr}

\date{March 25, 2022}%

\thanks{The authors were partially supported by CNRS Projet 80--Prime GAMM (Géométrie algébrique complexe/réelle et mécanique des matériaux).}


\begin{abstract}
  Generically,  a fully measured elasticity tensor has no material symmetry. For single crystals  with a cubic lattice, or for the aeronautics turbine blades superalloys such as Nickel-based CMSX-4, cubic symmetry is nevertheless expected. It is in practice necessary to compute the nearest cubic elasticity tensor to a given raw one. Mathematically formulated, the problem consists in finding the distance between a given tensor and the cubic symmetry stratum.

  It is known that closed symmetry strata (for any tensorial representation of the rotation group) are semialgebraic sets, defined by polynomial equations and inequalities. It has been recently shown that the closed cubic elasticity stratum is moreover algebraic, which means that it can be defined by polynomial equations only (without requirement to polynomial inequalities). We propose to make use of this mathematical property to formulate the distance to cubic symmetry problem as a polynomial (in fact quadratic) optimization problem, and to derive its quasi-analytical solution using the technique of Gröbner bases. The proposed methodology also applies to cubic Hill elasto-plasticity (where two fourth-order constitutive tensors are involved).
\end{abstract}

\maketitle

%

\section{Introduction}
\label{sec:intro}

Anisotropic elasto-plasticity theories introduce (at least) two fourth-order constitutive tensors, the Hooke and the Hill  tensors for instance. It is nowadays possible to measure/determine all their components
\cite{KRG1971,AHR1991,Art1993,Fra1995,FGB1998,Del2005,BAA2006,FGB1998,MTG2018}. These measured constitutive tensors are however generically triclinic (they have no material symmetry).

On the other hand, many materials (such as composite/engineered materials, single crystal superalloys or rocks) have an expected symmetry, most often due to their microstructure and their elaboration process. In practice, appealing to Curie principle (``the symmetries of the causes are to be found in the effects''), their constitutive tensors shall inherit the material symmetry (orthotropy, cubic or monoclinic symmetry for example), so that the natural question is
\emph{to determine the constitutive tensor with a given material symmetry the nearest to a given measured (triclinic) constitutive tensor.} This question has been extensively studied, from both the theoretical and numerical points of view, since the pioneering work of Gazis, Tadjbakhsh and Toupin \cite{GTT1963}, and subsequent works in the 90s \cite{AHR1991,Art1993,Fra1995,FGB1998}.
Most works focus on the elasticity tensor \cite{GTT1963,Fed1968,Fra1995,FBG1996,Hel1996,FGB1998,MN2006,KS2009,DKS2011}, a few ones
on the piezoelectricity tensor \cite{ZTP2013}. So far, we are not aware of some similar studies for the Hill plasticity tensor or the combination of several constitutive tensors.

Even if some analytical attempts exist \cite{Via1997,SMB2020,ADKD2021}, the \emph{distance to an elasticity symmetry class problem} is most often solved numerically, following \cite{FGB1998}, using the far from being injective parameterization of a symmetry class by its normal form $\bA$ (for instance \eqref{eq:VoigtCubic} for cubic symmetry \cite{Fed1968})
and a rotation $Q\in \SO(3)$,
\begin{equation*}
  \bE=Q\star \bA
  \qquad
  \left(
  \emph{i.e.}, \; E_{ijkl}=Q_{ip}Q_{jq}Q_{kr}Q_{ls} A_{pqrs}
  \right)
  ,
\end{equation*}
where $\star$ stands for the action of the rotation $Q$ on the tensor $\bA$ \cite{FV1996}. Letting $\bE_{0}$ be the given experimental (raw) elasticity tensor, one has then to minimize the squared norm
\begin{equation}\label{eq:distanceEla}
  \min_{\bE \, \text{of a given symmetry}} \norm{\bE_{0}-\bE}^{2}
  = \min_{Q,\bA} \norm{\bE_{0}-Q \star \bA}^{2}.
\end{equation}
The rotation $Q$ can be parameterized by the Euler angles or by a unit quaternion $\qq$ \cite{KS2008,KS2009}, allowing, in the second case, for the formulation of the considered distance problem as a \emph{polynomial optimization problem}. Indeed,
the function $\norm{\bE_{0}-Q(\qq) \star \bA}^{2}$ is then quadratic in $\bA$ and polynomial of degree 16 in $\qq$. Note, however, that a pair $(Q, \bA)$ is far from representing uniquely a tensor $\bE$. For instance, a cubic elasticity tensor $\bE$ is represented by 24 pairs $(Q, \bA)$ or 48 pairs $(\qq, \bA)$ (since $Q(\qq) = Q(-\qq)$). This means that we expect to find at least 24 global minima $(Q, \bA)$, or 48 global minima $(\qq, \bA)$ to problem~\eqref{eq:distanceEla}, which correspond to the same cubic elasticity tensor $\bE$. More generally, there are at least as many global minima $(Q, \bA)$ as there are symmetries of $\bA$ (24 for cubic symmetry).

Solving the distance to a given symmetry problem~\eqref{eq:distanceEla} is usually done numerically with, then, the risk to reach a local minimum instead of a global one~\cite{FGB1998,MN2006,DKS2015}. To overcome this difficulty, François and coworkers proposed to plot first pole figures for the given elasticity tensor $\bE_{0}$ \cite{Fra1995,FBG1996} (renamed plots of the monoclinic distance in~\cite{MN2006,KS2009}). Accordingly, they got an initial value for $\bE=Q\star\bA$, not too far from $\bE_{0}$, which was then optimized by a standard numerical (iterative) scheme.

Computational algebraic or semialgebraic optimization methods have been developed to find (directly) the global minimum of a multi-variable polynomial function (with polynomial constraints), using semidefinite programming for example
\cite{NN1992,Ali1995,VB1996,Tod2001,Lau2009,WSV2012,Las2015}.
On the other hand, there are nowadays symbolic computation methods to solve sets of polynomial equations, for instance the method of Gröbner bases \cite{Buc1965} (see also \cite{CLO2007,Stu1993}). These methods work well when the number of variables (\emph{i.e.}, of unknowns) is small and when the degree of the polynomials remains low \cite{Las2015}. The Gröbner basis method is available in the algebraic geometry software Macaulay2 \cite{Macau2} and in most Computer Algebra Systems. It does not make any numerical approximation if the coefficients of the considered polynomials are rational numbers and the Gröbner basis method can be seen as quasi-analytical. We use the prefix \emph{quasi} because at one step, after an exact variables elimination process, one has to solve a polynomial equation in one variable, the remaining equations becoming afterwards linear.

The elasticity symmetry classes have been characterized by polynomial equations and inequalities in \cite{OKDD2021} (see also \cite{ADD2020}, or \cite{AKP2014} for the case of harmonic fourth-order tensors), illustrating the mathematical property  that the closed $\SO(3)$-symmetry strata\footnote{A symmetry stratum is the set of all tensors which have the same symmetry class.} are semialgebraic sets~\cite{AS1981,AS1983,PS1985,Sch1989}. The necessary and sufficient conditions for a Hooke tensor to belong to one of the eight elasticity symmetry strata have been formulated using polynomial covariants (in a coordinate-free manner). For a Hooke tensor, the cubic stratum is characterized by quadratic equations~\cite[Theorem 10.3]{OKDD2021}. Therefore, one hopes to formulate the \emph{distance to cubic elasticity problem} as a quadratic optimization problem (of much lower degree than for the normal form/quaternion parameterization method) and expects a quasi-analytical solution using the Gröbner basis method. To succeed, one will simply have to derive first-order Euler--Lagrange equations for the corresponding quadratic optimization problem.

Cubic symmetry is of most importance for Ni-based single crystal superalloys, such as CMSX-4 \cite{FS1987,PT2006,Ree2006}, the material of aircrafts gas turbine blades (subject to (visco-)plasticity \cite{LC1985,BCCF2012}). Thanks to the harmonic decomposition \cite{Bac1970,Spe1970,Cow1989,Bae1993}, the geometry of cubic fourth order tensors is now well understood. This will make it possible to formulate the calculation of the distance to cubic symmetry as a polynomial optimization problem, not only for a single elasticity tensor, but also for a pair $(\bE, \bP)$ of two fourth-order constitutive tensors. Here, $\bE$ is understood as the Hooke (elasticity) tensor and $\bP$ as the Hill (plasticity) tensor. Indeed, we shall see that this pair of tensors is at least cubic if and only if the harmonic second-order components of $\bE$ and $\bP$ vanish and if their harmonic fourth-order components are at least cubic and proportional.

We will make use of the reformulation of the \emph{distance to cubic elasticity} as a \emph{quadratic optimization problem},
in order to solve it quasi-analytically. In practice, this will be done thanks to the \emph{theory of Gröbner bases}.
We will  take advantage of the fact that the material parameters, such as the components of an experimental elasticity tensor, are measured with only a few significant digits to work with rational coefficients polynomials. This point is of main importance in the resolution of a system of polynomial equations by the obtention of a Gröbner basis (see remark \ref{rem:coeffQ} in the Appendix).

The paper is organized as follows. The Euler--Lagrange method for solving constrained optimization problems is briefly presented in section \ref{sec:EL-method}. Background materials on cubic  constitutive tensors are recalled in \autoref{sec:cubic-laws} and \autoref{sec:cubic-pair}. The problem of the distance to cubic elasticity is formulated as a polynomial (quadratic) optimization problem in \autoref{sec:distance-cubic-elasticity} and solved thanks to the theory of Gröbner bases in \autoref{sec:num-cubic-elasticity}. The extension to the pair $(\bE, \bP)$ of the Hooke and Hill tensors is described in \autoref{sec:distance-cubic-elasto-plasticity} and \autoref{sec:num-cubic-elasto-plasticity}. Finally, in \autoref{sec:recovering-cubic-normal-form}, we explain how to compute a natural cubic basis for any given cubic Hooke tensor. To be self-contained, a summary of Gröbner bases methods for algebraic elimination is provided in \autoref{sec:groebner}.

\subsection*{Notations}

We are working in orthonormal bases, so that we do not have to distinguish between covariant and contravariant tensors.
The tensor product is denoted by $\otimes$. An \emph{harmonic tensor} is a traceless totally symmetric tensor.
The space of harmonic tensors of order $n$ will be denoted by $\HH^{n}(\RR^{3})$ or simply $\HH^n$. It is a subspace of dimension $2n + 1$ of the vector space $\Sym^{n}(\RR^{3})$, the space of totally symmetric tensors $\bS=\bS^{s}$ of order $n$ (where $(\cdot)^{s}$ is the symmetrization operator).

Let $\bS=\bS^{s}\in \Sym^{p}(\RR^{3})$ (of order $p$) and $\bT=\bT^{s}\in \Sym^{q}(\RR^{3})$ (of order $q$) be two totally symmetric tensors. The \emph{totally symmetric tensor product} $\odot$ is defined by
\begin{equation*}
  \bS\odot \bT :=(\bS\otimes \bT)^{s}\in \Sym^{p+q}(\RR^{3}).
\end{equation*}
It is a totally symmetric tensor (of order $p+q$). The \emph{generalized cross product} between two totally symmetric tensors, which was introduced in~\cite{OKDD2021}, is defined by
\begin{equation}\label{eq:SxT}
  \bS\times \bT := (\bT\cdot \lc \cdot \bS)^{s}\in \Sym^{p+q-1},
\end{equation}
where $\lc=(\varepsilon_{ijk})$ is the Levi-Civita tensor. In components, it is written as
\begin{equation*}
  (T_{i_{1} \dotsc i_{p-1} k} \, \varepsilon_{k i_{p} l}\,  S_{l i_{p+1} \dotsc i_{p+q-1}})^{s}.
\end{equation*}

A dot denotes a contraction between two tensors and several dots, several contractions. For instance
\begin{gather*}
  (\ba\cdot \bb)_{ij} = a_{ik}b_{kj}, \qquad \ba: \bb=a_{ij}b_{ij},
  \\
  (\bH : \ba)_{ij} = H_{ijkl}a_{kl}, \qquad (\bH : \bK)_{ijkl} = H_{ijpq}K_{pqkl}, \qquad (\bH \3dots \bK)_{ij} = H_{ipqr}K_{pqrj},
\end{gather*}
where $\ba$, $\bb$ are second-order tensors and $\bH$, $\bK$, fourth-order tensors. The usual abbreviations $\bH^{2}=\bH : \bH$ and $\bH^{3}=\bH : \bH:\bH$ shall also be used.

\section{The Euler--Lagrange method for polynomial functions and constraints}
\label{sec:EL-method}

The simplest method to solve a minimization problem for a polynomial function $f$, defined on $\RR^{n}$, is probably the \emph{Euler--Lagrange method}, which consists in looking for its critical points. The critical points of $f$ are solutions of a system of algebraic equations which may be solved using \emph{Gröbner bases} for example (see \autoref{sec:groebner}). When, moreover, polynomial algebraic constraints $g(\xx) = 0$ are involved, where
\begin{equation*}
  g : \RR^{n} \to \RR^{p},
\end{equation*}
is a smooth vector-valued function, the \emph{method of Lagrange multipliers} can be used~\cite{BD1990,Kal2009,Laf2015}. In geometric terms, the constraint problem means that we seek for critical points of the restriction of $f$ to the submanifold of $\RR^{n}$
\begin{equation}\label{eq:submersion-condition}
  S := \set{\xx \in \RR^{n};\; g(\xx)=0}.
\end{equation}
This requires that the constraint function, $g : \RR^{n} \to \RR^{p}$, is a \emph{submersion} on $S =g^{-1}(0)$, which means that the linear tangent map (\emph{i.e.}, here the Jacobian matrix)
\begin{equation*}
  T_{\xx}g: \RR^{n} \to \RR^{p}
\end{equation*}
is of maximal rank $p$ at each point $\xx \in S$ (which requires that $n \ge p$). This condition ensures that $S$ is a smooth \emph{submanifold} of $\RR^n$ of dimension $n-p$~\cite{Laf2015}. In that case, one can show, using the implicit function theorem, that the solutions of the constrained problem
\begin{equation}\label{eq:constrained-problem}
  \min_{\xx}\, f(\xx) \quad \text{with} \quad g(\xx)=0,
\end{equation}
are critical points of the function
\begin{equation}\label{eq:Fxlambda}
  F(\xx,\blambda) = f(\xx) + (\blambda,g(\xx)),
\end{equation}
where $(\cdot,\cdot)$ is the duality bracket on $\RR^{p}$ and the dual variable $\blambda$ is the \emph{Lagrange multiplier}. A proof of this fact can be found in~\cite[Theorem 3.5.27]{AMR1988}.

\begin{rem}
  Note however that the extrema of $f$ are in general saddle points of $F$~\cite{Kal2009}. Therefore, one should not make the false statement that the minimization of the constrained problem \eqref{eq:constrained-problem} is equivalent to the minimization of the function~\eqref{eq:Fxlambda}.
\end{rem}

The critical points of~\eqref{eq:Fxlambda} are the solutions of the algebraic system
\begin{equation}\label{eq:EL-submersion}
  \begin{cases}
    \displaystyle \frac{\partial F}{\partial \xx} = \frac{\partial f}{\partial \xx} + \left(\blambda,\frac{\partial g}{\partial \xx}\right) = 0,
    \\
    \displaystyle \frac{\partial F}{\partial \blambda} = g = 0.
  \end{cases}
\end{equation}
These equations are referred to as (first-order) \emph{Euler--Lagrange equations with constraints}.

In practice, however, the problem is not that simple. In several problems, the set $S = g^{-1}(0)$ contains some point $\xx$ at which $g$ is not a submersion. Worse, in the following example, which concerns the distance of a deviatoric second order tensors $\xx=\bh\in \HH^{2}$ to transverse isotropy, the gradient of $g$ is singular at each point $\bh$ of $S$.

\begin{exmp}[The transversely isotropic (closed) strata in $\HH^{2}$]
  It is the vector subspace $S$ of $\HH^{2}$ of deviatoric tensors which have at least two identical eigenvalues. The set $S$ is defined implicitly by the polynomial equation
  \begin{equation*}
    g(\bh):=\left(\tr \bh^{2}\right)^{3}-6 \left(\tr \bh^{3}\right)^{2}=0,
    \qquad \bh \in \HH^{2}.
  \end{equation*}
  Unfortunately, the gradient of $g$ in $\HH^{2}$
  \begin{equation*}
    \grad_{\bh} g = 6\left((\tr \bh^{2})^{2}\bh - 6(\tr \bh^{3})(\bh^{2})^{\prime} \right)
  \end{equation*}
  vanishes identically on the set $S = g^{-1}(0)$, since, when $\bh \in \HH^{2}$ is transversely isotropic, we have
  \begin{equation*}
    (\bh^{2})^{\prime} = \frac{\tr \bh^{3}}{\tr \bh^{2}} \, \bh
  \end{equation*}
  and thus
  \begin{equation*}
    \grad_{\bh} g  = \frac{6g(\bh)}{\tr \bh^{2}}\bh = 0.
  \end{equation*}
\end{exmp}

Fortunately, all situations are not as bad as in this example but singularities may still exist. In the following examples, concerning respectively cubic fourth-order harmonic tensors and elasticity tensors, of main interest for the present work, the set $S$ is defined by a mapping $g$ which is a submersion on a big open subset of $S$, but not on all of $S$.

\begin{exmp}[The cubic (closed) strata in $\HH^{4}$]
  It is the vector subspace $S$ of $\HH^{4}$ of fourth-order harmonic tensors which are at least cubic. It was shown in~\cite[Theorem 9.3]{OKDD2021}, that this set can be defined as
  \begin{equation*}
    S = \set{\bH \in \HH^{4};\; g(\bH)=0},
  \end{equation*}
  where
  \begin{equation*}
    g: \HH^{4} \to \HH^{2}
  \end{equation*}
  is a polynomial mapping of degree $2$. One can check that $g$ is a submersion at each cubic tensor $\bH \ne 0$, but not at $\bH=0$, which is a singular point.
\end{exmp}

\begin{exmp}[The cubic (closed) strata in $\Ela$]
  It is the vector subspace $S$ of $\Ela$ of elasticity tensors which are at least cubic. It was shown in~\cite[Theorem 10.2]{OKDD2021}, that this set can be defined as
  \begin{equation*}
    S = \set{\bE \in \Ela;\; g(\bE)=0},
  \end{equation*}
  where
  \begin{equation*}
    g: \Ela \to \HH^{2} \oplus \HH^{2} \oplus \HH^{2}
  \end{equation*}
  is a polynomial mapping. One can check that $g$ is a submersion at each point $\bE \in S$ if $\bE$ is cubic, but not if $\bE$ is isotropic.
\end{exmp}

\begin{rem}
  In \cite{GD2014} and \cite{WYZ2019} are proposed some algorithms to solve the constrained problem \eqref{eq:constrained-problem} even when the polynomial mapping $g$ is not a submersion, under some further hypotheses (the implementation of the algorithm proposed in \cite{GD2014} is available on the webpage of the first author as a Maple library). These algorithms involve the notions of \emph{nonsingular} and \emph{singular} points of the real algebraic set $S = \set{\xx \in \RR^{n};\; g(\xx)=0}$. Under some hypotheses on $S$ and on the polynomial coordinate functions $(g_1,\dotsc,g_p)$ of $g$, a point $\xx$ of $S$ is said to be nonsingular if the Jacobian matrix of $g$ at $\xx$ is of rank $n-d$, where $d$ is the {\it dimension} of the real algebraic set $S$ (which is by definition the so-called {\it Krull dimension} of the ring of polynomial functions on $S$), otherwise $\xx$ is said to be singular. If the real algebraic set $S$ has no singular point, it is said to be nonsingular and, in this case, $S$ is a smooth submanifold of $\RR^n$ of dimension $d$ (precise definitions and properties can be found in \cite{BCM}). However, the correctness of the subroutine {\it GenCritValues} of \cite[Section 4]{WYZ2019} is, as far as we understand, not clear for us since it refers to an algorithm of \cite{JelKur2005} which is carried out on {\it complex} algebraic sets, not on real ones.
\end{rem}

\section{Cubic elasticity tensors}
\label{sec:cubic-laws}

The space of elasticity tensors~\cite{FV1996}, denoted by $\Ela$, is the space of fourth-order tensors $\bE$ with the following index symmetries
\begin{equation*}
  E_{ijkl}=E_{jikl}=E_{ijlk}=E_{klij}.
\end{equation*}
$\Ela$ is a vector space of dimension 21 and an elasticity tensor $\bE \in \Ela$ can be represented in Voigt notation by the matrix
\begin{equation*}
  [\bE]=
  \begin{pmatrix}
    E_{1111} & E_{1122} & E_{1133} & E_{1123} & E_{1113} & E_{1112} \\
    E_{1122} & E_{2222} & E_{2233} & E_{2223} & E_{1223} & E_{1222} \\
    E_{1133} & E_{2233} & E_{3333} & E_{2333} & E_{1333} & E_{1233} \\
    E_{1123} & E_{2223} & E_{2333} & E_{2323} & E_{2331} & E_{2312} \\
    E_{1113} & E_{1223} & E_{1333} & E_{2331} & E_{1313} & E_{3112} \\
    E_{1112} & E_{1222} & E_{1233} & E_{2312} & E_{3112} & E_{1212}
  \end{pmatrix}.
\end{equation*}
If $\bE$ has \emph{at least} the cubic symmetry, there exists an orthonormal basis $(\ee_{i})$ (called the natural basis or the cubic basis), in which $\bE$ has the so-called \emph{cubic normal form} in Voigt representation
\begin{equation}
  \label{eq:VoigtCubic}
  [\bE]=\begin{pmatrix}
    E_{1111} & E_{1122} & E_{1122} & 0        & 0        & 0        \\
    E_{1122} & E_{1111} & E_{1122} & 0        & 0        & 0        \\
    E_{1122} & E_{1122} & E_{1111} & 0        & 0        & 0        \\
    0        & 0        & 0        & E_{1212} & 0        & 0        \\
    0        & 0        & 0        & 0        & E_{1212} & 0        \\
    0        & 0        & 0        & 0        & 0        & E_{1212}
  \end{pmatrix}_{(\ee_{1}, \ee_{2}, \ee_{3})}.
  \quad
\end{equation}
If $E_{1111} - E_{1122} - 2E_{1212}=0$, then $\bE$ is isotropic. Otherwise, it is cubic. One may point out the cubic symmetry group ($\octa$) and write $\bE_\octa$ for the normal form of a cubic tensor $\bE$.

If $E$, $\nu$ and $G$, respectively denote the Young modulus, the Poisson ratio, and the shear modulus of a material, the engineer's expressions for the $E_{ijkl}$ are
\begin{equation*}
  \begin{cases}
    E_{1111}=\displaystyle\frac{(1-\nu)E}{1-\nu-2 \nu ^{2}},
    \\
    E_{1122}=\displaystyle\frac{\nu E}{1-\nu-2 \nu ^{2}},
    \\
    E_{1212}=G,
  \end{cases}
\end{equation*}
and $G\neq E/2(1+\nu)$ when $\bE$ is cubic. In intrinsic notations, a cubic elasticity tensor $\bE$ can be rewritten as in~\cite{FV1996}
\begin{equation}\label{eq:Eharm}
  \bE= 2\mu\, \mathbf{I} + \lambda \Id \otimes \Id+\bH,
  \qquad
  \bH\neq 0,
\end{equation}
where $\bI$ is the fourth order tensor with components $I_{ijkl}=\frac{1}{2} (\delta_{ik}\delta_{jl}+\delta_{il}\delta_{jk})$
and $\bH\in \HH^{4}$ is a cubic fourth-order harmonic tensor\footnote{\emph{i.e.} totally symmetric, $\bH=\bH^{s}$, and traceless, $\tr_{ij} \bH=0$.}. Here,
\begin{equation*}
  \lambda = \frac{1}{15}(2 \tr(\tr_{12} \bE)-\tr(\tr_{13} \bE))=\frac{1}{5}(E_{1111} - 2E_{1212} + 4E_{1122}),
\end{equation*}
and
\begin{equation*}
  \mu = \frac{1}{30}(- \tr(\tr_{12} \bE)+3\tr(\tr_{13} \bE))
  =\frac{1}{5}(E_{1111}+3E_{1212} - E_{1122}),
\end{equation*}
are the Lamé constants.

A fourth-order harmonic tensor
has  $9$ independent components. It can always be parameterized as (in Voigt notation, see \cite{DD2016,OKDD2018}),
\begin{equation}\label{eq:Hcubic}
  [\bH]=\begin{pmatrix} \Lambda_{2}+\Lambda_{3} & -\Lambda_{3} & -\Lambda_{2} & -X_{1} & Y_{1}+Y_{2} & -Z_{2} \\ -\Lambda_{3} & \Lambda_{3}+\Lambda_{1} & -\Lambda_{1} & -X_{2} & - Y_{1} & Z_{1}+Z_{2} \\ -\Lambda_{2} & -\Lambda_{1} & \Lambda_{1}+\Lambda_{2} & X_{1}+X_{2} & -Y_{2} & -Z_{1} \\ -X_{1} & -X_{2} & X_{1}+X_{2} & -\Lambda_{1} & -Z_{1} & -Y_{1} \\ Y_{1}+Y_{2} & - Y_{1} & -Y_{2} & -Z_{1} & -\Lambda_{2} & -X_{1} \\ -Z_{2} & Z_{1}+Z_{2} & -Z_{1} & -Y_{1} & -X_{1} & -\Lambda_{3}  \end{pmatrix}
  .
\end{equation}
If $\bH$ has \emph{at least} the cubic symmetry, there exists an orthonormal basis $(\ee_{i})$, in which (in Voigt notation, see \cite{AKP2014}):
\begin{equation}\label{eq:VoigtCubicH}
  [\bH]= \delta \left(
  \begin{array}{cccccc}
      8  & -4 & -4 & 0  & 0  & 0  \\
      -4 & 8  & -4 & 0  & 0  & 0  \\
      -4 & -4 & 8  & 0  & 0  & 0  \\
      0  & 0  & 0  & -4 & 0  & 0  \\
      0  & 0  & 0  & 0  & -4 & 0  \\
      0  & 0  & 0  & 0  & 0  & -4 \\
    \end{array}
  \right)_{(\ee_{1}, \ee_{2}, \ee_{3})}
  \!\!\!\!\!\!\!\!\!\!\!\! ,
  \qquad
  \delta =\frac{1}{4}\left(\mu-G\right)
  ,
\end{equation}
with $\delta=0$ when $\bH$ is isotropic and  $\delta\neq 0$ when it is cubic.

\begin{rem}
  The decomposition \eqref{eq:Eharm} of $\bE$ into $\lambda$, $\mu$, and $\bH$ (with $\bH$ cubic), is the so-called harmonic decomposition of a cubic elasticity tensor (see \cite{Bac1970,Cow1989}).
\end{rem}

The generalized Lamé constants $\lambda=\lambda(\bE), \mu=\mu(\bE)$ are two (linear) invariants of $\bE$. The scalar $\delta=\delta(\bE)$ is a (rational) invariant of the cubic elasticity tensor $\bE$. Indeed, one has then~\cite[Section 5.1]{AKP2014}:
\begin{equation}\label{eq:delta}
  \delta=\frac{J_{3}}{4J_{2}},
\end{equation}
where
\begin{equation}\label{eq:J2J3}
  J_{2} = \norm{\bH}^{2}=\tr (\tr_{13} \bH^{2}) = H_{ijkl}H_{ijkl},
  \quad
  \text{and}
  \quad
  J_{3} = \tr (\tr_{13} \bH^{3}) = H_{ijkl}H_{klpq}H_{pqij},
\end{equation}
are two polynomial invariants of $\bH$ (first introduced in~\cite{BKO1994}). The Euclidean squared norm of the cubic elasticity tensor $\bE$ is then
\begin{equation*}
  \norm{\bE}^{2} = 3 \left(3\lambda^{2} + 4 \lambda \mu + 8 \mu^{2}\right) + 480 \delta^{2}.
\end{equation*}
When evaluated on \eqref{eq:Hcubic}, the invariants $J_{2}$ and $J_{3}$ can be expressed as
\begin{align}\label{eq:J2detail}
  J_{2} & = 2 \Big(4 \Lambda_{1}^{2} + \Lambda_{1} \Lambda_{2} + \Lambda_{1} \Lambda_{3} + 4 \Lambda_{2}^{2}+\Lambda_{2} \Lambda_{3} + 4 \Lambda_{3}^{2} + 8 X_{1}^{2} + 4 X_{1} X_{2} + 4 X_{2}^{2}
  \\
  \nonumber
        & \quad +8 Y_{1}^{2}+4 Y_{1} Y_{2}+4 Y_{2}^{2}+8 Z_{1}^{2}+4 Z_{1} Z_{2}+4 Z_{2}^{2}\Big),
  \\
  \label{eq:J3detail}
  J_{3} & = \, 6 \Big(\Lambda_{1}^{2} \Lambda_{2} + \Lambda_{1} \Lambda_{2}^{2} + \Lambda_{3}^{2} (\Lambda_{1} + \Lambda_{2}) - 3 X_{1}^{2} (\Lambda_{1} + \Lambda_{3}) - 2 \Lambda_{3} X_{1} X_{2}
  \\
  \nonumber
        & \quad + 4 X_{1} (\Lambda_{2} X_{2}+(Y_{1}+Y_{2}) (Z_{1}+Z_{2}))+X_{2}^{2} (\Lambda_{2}+\Lambda_{3})+4 X_{2} (Z_{1} (Y_{1}+Y_{2})+Y_{1} Z_{2})
  \\
  \nonumber
        & \quad -3 \Lambda_{1} Y_{1}^{2}-3 \Lambda_{2} Y_{1}^{2}+\Lambda_{3} \left(\Lambda_{1}^{2}-\Lambda_{1} \Lambda_{2}+\Lambda_{2}^{2}+4 Y_{1} Y_{2}+Y_{2}^{2}-3 Z_{1}^{2}\right)-2 \Lambda_{1} Y_{1} Y_{2}
  \\
  \nonumber
        & \quad + \Lambda_{1} Y_{2}^{2} - 3 \Lambda_{2} Z_{1}^{2} + 4 \Lambda_{1} Z_{1} Z_{2} - 2 \Lambda_{2} Z_{1} Z_{2} + \Lambda_{1} Z_{2}^{2} + \Lambda_{2} Z_{2}^{2}\Big).
\end{align}

\begin{rem}\label{rem:normal-form-hooke}
  Given a cubic elasticity tensor $\bE^{*}=(E_{ijkl}^{*})$, expressed in an arbitrary basis, the calculation of its normal form $\bE=\bE_{\octa}$ (of Voigt representation~\eqref{eq:VoigtCubic}) is straightforward (using~\eqref{eq:VoigtCubicH} within~\eqref{eq:Eharm}). Indeed, the normal form~\eqref{eq:VoigtCubic} is recovered from the calculation of $\lambda$, $\mu$ and $\delta=J_{3}/4J_{2}$ by the above formulas with
  \begin{equation}\label{eq:CubicDirect}
    E_{1111}=2\mu + \lambda+ 8 \delta,
    \qquad
    E_{1122} =  \lambda-4 \delta,
    \qquad
    E_{1212}= \mu - 4 \delta,
    \qquad
  \end{equation}
  where the invariants $\lambda$, $\mu$, $J_{2}$, $J_{3}$ and $\delta$ are evaluated on $\bE^*$.
\end{rem}

The covariant characterization of the elasticity symmetry classes by polynomial equations (and inequalities) has been performed recently,
in~\cite[theorem 10.2]{OKDD2021}. The case of the cubic symmetry is recalled as theorem~\ref{thm:Ecubic} below. We denote by
\begin{equation*}
  \ba'=\ba-\frac{1}{3} \tr(\ba)\, \Id,
\end{equation*}
the deviatoric part of a second-order tensor $\ba$, and by $\bE^{s}$, the totally symmetric part of $\bE$, with components
\begin{equation*}
  E^{s}_{ijkl} = \frac{1}{3}(E_{ijkl}+E_{ikjl}+E_{iljk}).
\end{equation*}

\begin{thm}[Olive et al (2021)]\label{thm:Ecubic}
  Let $\bE$ be an elasticity tensor,
  \begin{equation*}
    \bd=\tr_{12} \bE
    \quad \text{and} \quad
    \bv=\tr_{13} \bE,
  \end{equation*}
  respectively, the dilatation and the Voigt second-order tensors,
  \begin{equation*}
    \lambda = \frac{1}{15}(2 \tr\bd-\tr\bv)
    \quad \text{and} \quad
    \mu = \frac{1}{30}(3\tr\bv - \tr\bd),
  \end{equation*}
  the Lamé constants,
  \begin{equation}\label{eq:H}
    \bH=\bE^{s}-(2\mu+ \lambda) \Id \odot \Id -
    \frac{2}{7}\Id \odot ( \bd^{\prime} +2  \bv^{\prime})
  \end{equation}
  and
  \begin{equation*}
    \bd_{2}=\bH\3dots \bH,
  \end{equation*}
  with components $(\bd_{2})_{ij}=H_{ipqr}H_{pqrj}$. Then, $\bE$ is cubic if and only if
  \begin{equation*}
    \bd'=\bv'=0,
    \qquad
    \bd_{2}'=0,
    \quad \text{and} \quad
    J_{2}=\tr \bd_{2} \neq 0.
  \end{equation*}
\end{thm}

\begin{rem}\label{rem:HrmDecompNorm2}
  The decomposition of any elasticity tensor provided by the above formulas,
  \begin{equation*}
    \bE=(\lambda, \mu, \bd', \bv', \bH),
  \end{equation*}
  into the harmonic components $\lambda, \mu\in \HH^0$, $\bd', \bv'\in \HH^{2}$ and $\bH\in \HH^{4}$, is the so-called harmonic decomposition of $\bE$~\cite{Bac1970,Spe1970,Cow1989}.
\end{rem}

\section{Cubic pair of elasticity-like tensors}
\label{sec:cubic-pair}

There exist constitutive laws (for instance, anisotropic elasto-plasticity \cite{Hil1948}) involving two fourth-order constitutive tensors. The question of the characterization of all the symmetry classes of a pair
\begin{equation*}
  (\bE, \bF) \in \Ela \times \Ela
\end{equation*}
of elasticity-like tensors seems to be an open one. Nevertheless, this question has a relatively simple answer in the cubic symmetry case, thanks to the harmonic decompositions of both tensors $\bE$ and $\bF$,
\begin{equation*}
  \bE=(\lambda, \mu, \bd', \bv', \bH)
  \quad \text{and} \quad \bF=(\ell, m, \be', \bw', \bK),
\end{equation*}
and by recalling that the symmetry group $G_{(\bE, \bF)}$ of the pair $(\bE, \bF)$ is the intersection of the symmetry groups of its harmonic components \cite{FV1996},
\begin{equation*}
  G_{(\bE, \bF)}=G_{\bE} \cap G_{\bF}=G_{\bd'}  \cap  G_{\bv'}  \cap  G_{\bH}  \cap  G_{\be'}  \cap  G_{\bw'}  \cap  G_{\bK}.
\end{equation*}
As an harmonic (deviatoric) cubic second-order  tensor is isotropic and therefore vanishes (so that $G_{(\bE, \bF)}=G_{\bH} \cap G_{\bK}$), and as the normal form of an harmonic cubic fourth-order  tensor is one-dimensional, the pair of elasticity-like fourth order tensors $(\bE, \bF)$ is cubic if and only if its harmonic second-order components vanish and its harmonic fourth-order components are cubic and proportional. By theorem \ref{thm:Ecubic} we get the following result.

\begin{thm}\label{thm:EFcubic}
  Let $\bE=(\lambda, \mu, \bd', \bv', \bH)\in \Ela$ and $\bF=(\ell, m, \be', \bw', \bK)\in \Ela$ be two elasticity-like fourth-order tensors, and
  \begin{equation*}
    \bd_{2}(\bH)=\bH\3dots \bH,
    \qquad
    \bd_{2}(\bK)=\bK\3dots \bK,
  \end{equation*}
  be the quadratic covariants of their respective harmonic fourth-order components $\bH$ and $\bK$. Then, the pair $(\bE, \bF)$ is cubic if and only if
  \begin{equation*}
    \bd'=\bv'=\be'=\bw'=0,
  \end{equation*}
  and either
  \begin{equation*}
    (a) \qquad \bd_{2}'(\bH)=0
    \quad \text{and} \quad
    \bH=k \bK\neq 0,
  \end{equation*}
  or
  \begin{equation*}
    (b) \qquad \bd_{2}'(\bK)=0
    \quad \text{and} \quad
    \bK=k \bH\neq 0,
  \end{equation*}
  with $k\in \RR$.
\end{thm}

\section{Distance to cubic elasticity as a quadratic optimization problem}
\label{sec:distance-cubic-elasticity}

It is possible to reformulate the  distance to cubic symmetry problem into a \emph{quadratic optimization problem}
\begin{equation*}
  \min_{\bE} \norm{\bE_{0}-\bE}^{2} \quad \text{with} \;  \bE \; \text{at least cubic},
\end{equation*}
since the function to be minimized $\norm{\bE_{0}-\bE}^{2}$ and the constraint ``$\bE$ at least cubic''
(equivalent to $\bd'=\bv'=0$ and $\bd_{2}'=0$ by theorem \ref{thm:Ecubic}) are expressed by quadratic polynomials in $\bE$.
This 21-dimensional problem can be further reduced to a 9-dimensional optimization problem in the harmonic component $\bH\in \HH^{4}$ of $\bE$ only.

To achieve this reduction, we perform the harmonic decompositions of both the given and the sought tensors $\bE_0$ and $\bE$ (see remark \ref{rem:HrmDecompNorm2} and \eqref{eq:Eharm}),
\begin{equation*}
  \bE_{0}=(\lambda_{0}, \mu_{0}, \bd_{0}', \bv_{0}', \bH_{0})
  \quad \text{and} \quad
  \bE=(\lambda, \mu, 0, 0, \bH).
\end{equation*}
Then, using the formula
\begin{equation}\label{eq:normE2}
  \norm{\bE}^{2} =3 \left(3 \lambda ^{2}+4 \lambda  \mu +8 \mu ^{2}\right) +  \frac{2}{21} \norm{\bd^{\prime}+2 \bv^{\prime}}^{2} +\frac{4}{3}\norm{\bd^{\prime}- \bv^{\prime}}^{2}+ \norm{\bH}^{2},
\end{equation}
for the Euclidean squared norm $\norm{\bE}^{2}:=\bE::\bE$
of an elasticity tensor
$\bE=(\lambda, \mu, \bd', \bv', \bH) $,
we get
\begin{align*}
  f(\bE)=\norm{\bE_{0}-\bE}^{2} & =3 \left(3 (\lambda_{0}-\lambda) ^{2}+4 (\lambda_{0}-\lambda) (\mu_{0}-\mu) +8 (\mu_{0}-\mu) ^{2}\right)
  \\
                                & \qquad
  +  \frac{2}{21} \norm{\bd_{0}^{\prime}+2 \bv_{0}^{\prime}}^{2} +\frac{4}{3}\norm{\bd_{0}^{\prime}- \bv_{0}^{\prime}}^{2}+ \norm{\bH_{0}-\bH}^{2},
\end{align*}
whose minimum for $\bE=\bE^*$ cubic is obtained for $\lambda=\lambda_{0}$, $\mu=\mu_{0}$ and $\bH$ cubic. We have therefore
\begin{equation*}
  \bE^*=2\mu_{0}\, \mathbf{I} + \lambda_{0} \Id \otimes \Id+\bH^*,
\end{equation*}
with $\bH^*\in \HH^{4}$ solution of the quadratic optimization problem
\begin{equation}\label{eq:PBenH}
  \min_{\bH} \norm{\bH_{0}-\bH}^{2} \quad \text{with} \quad g=\bd_{2}'=0,
\end{equation}
and the five scalar constraints in \eqref{eq:constrained-problem}
\begin{equation*}
  (\bd_{2}')_{11}=0, \quad (\bd_{2}')_{22}=0, \quad (\bd_{2}')_{12}=0, \quad (\bd_{2}')_{13}=0, \quad \text{and} \quad (\bd_{2}')_{23}=0,
\end{equation*}
are indeed quadratic in $\bH$.

The optimum is cubic if $\bH^*\neq 0$, with then the \emph{distance} and the \emph{relative distance} to cubic symmetry respectively
equal to
\begin{equation*}
  d(\bE_{0}, \text{cubic symmetry})=\norm{\bE_{0}-\bE^{*}}=\sqrt{
    \frac{2}{21} \norm{\bd_{0}^{\prime}+2 \bv_{0}^{\prime}}^{2} +\frac{4}{3}\norm{\bd_{0}^{\prime}- \bv_{0}^{\prime}}^{2}+ \norm{\bH_{0}-\bH^{*}}^{2}},
\end{equation*}
and
\begin{equation*}
  \frac{d(\bE_{0}, \text{cubic symmetry})}{\norm{\bE_{0}}}=
  \frac{\norm{\bE_{0}-\bE^{*}}}{\norm{\bE_{0}}}.
\end{equation*}

In order to apply the Euler--Lagrange method to our constrained optimization problem \eqref{eq:PBenH}, we have to check (see  \autoref{sec:EL-method}) that the smooth mapping
\begin{equation*}
  g:\; \HH^{4} \to \HH^{2} ,
  \qquad
  \bH \mapsto \bd_{2}'=(\bH\3dots \bH)',
\end{equation*}
is a submersion for all cubic tensors $\bH\in \HH^{4}$ (\textit{i.e.}, that the Jacobian matrix $T_{\bH}g: \HH^{4} \to \HH^{2}$ is of maximum rank 5, for each cubic tensor $\bH$). This is indeed the case. To show this, we observe that the mapping $\bH \mapsto g(\bH) = (\bH\3dots \bH)'$ is covariant, meaning that
\begin{equation*}
  g(Q\star \bH) = Q\star g(\bH)
\end{equation*}
for every rotation $Q$. Therefore, the rank of $T_{\bH}g$ is equal to the rank of $T_{Q \star \bH}g$ for every rotation $Q$ and it is enough to compute this rank when $\bH$ is the cubic normal form~\eqref{eq:VoigtCubicH}, which is 5. Note however that $g$ is not a submersion when $\bH=0$ (\emph{i.e.}, when $\bH$ is isotropic).

The Euler--Lagrange method further reduces the distance problem
(at given $\bH_{0}$),
\begin{equation*}
  \min_{g(\bH)=0} f(\bH),
  \qquad
  f(\bH)=\norm{\bH_{0}-\bH}^{2},
  \qquad
  g(\bH)=\bd'_{2},
\end{equation*}
to the determination of the critical points of the polynomial function
\begin{equation*}
  F(\bH, \blambda)=\norm{\bH_{0}-\bH}^{2}+\blambda:g(\bH),
\end{equation*}
with $\bH\in \HH^{4}$ an harmonic fourth-order tensor and where the Lagrange multiplier $\blambda\in \HH^{2}$ is an harmonic (deviatoric) second-order tensor.

The differential of $F$ with respect to $\bH$ is given by
\begin{align*}
  \dd F.\delta \bH & =2(\bH-\bH_{0})::\delta \bH+\blambda:(\bH\3dots \delta \bH+\delta \bH \3dots \bH) \\
                   & =2(\bH-\bH_{0})::\delta \bH+2\, \bS(\blambda) :: \delta\bH ,
\end{align*}
thanks to the equalities $\blambda:\bd'_{2}=\blambda:\bd_{2}=\blambda:(\bH \3dots  \bH)$, where
\begin{equation*}
  \bS(\blambda) :=\frac{1}{2} \grad_{\bH} (\blambda:\bd_{2}') = (\bH\cdot \blambda)^{s \prime} \; \in \HH^{4},
\end{equation*}
is the fourth-order harmonic part of the tensor $(\bH\cdot \blambda)^{s}$ (of components $H_{ijkp}\lambda_{pl}$).
It can be computed using for example Eq. \eqref{eq:H}, or using directly the harmonic decomposition of totally symmetric tensors \cite{Spe1970}\cite[Section 2.2]{OKDD2018}, with here
\begin{equation}\label{eq:trS}
  \tr (\bH\cdot \blambda)^{s}= \frac{1}{2}\, \bH:\blambda
  \quad \text{and} \quad
  \tr \tr \left( \bS(\blambda)\right)=\tr \tr (\bH\cdot \blambda)^{s}=0 ,
\end{equation}
so that (introducing the symmetrized tensor product $\odot$)
\begin{align}\label{eq:Sprim}
  \bS(\blambda)=(\bH\cdot \blambda)^{s}- \frac{3}{7} \Id \odot (\bH:\blambda).
\end{align}

Therefore, the Euler--Lagrange equations $\displaystyle\frac{\partial F}{\partial \bH}=0$ and $\displaystyle\frac{\partial F}{\partial \blambda}=0$ reduce to the system of equations
\begin{align}\label{eq:systeme}
  \begin{cases}
    \bH-\bH_{0}+\bS(\blambda)=0 & \textit{(9 scalar equations)}, \\
    \bd'_{2}=0                  & \textit{(5 scalar equations)},
  \end{cases}
\end{align}
in the 9 independent components $H_{ijkl}$ of $\bH\in \HH^{4}$ and the 5 independent components $\lambda_{ij}$ of $\blambda\in \HH^{2}$.

The system \eqref{eq:systeme} can be further simplified by extracting from the equality $\bH-\bH_{0}+\bS(\blambda)=0$ some linear equations in $\bH$.

\begin{lem}\label{lem:sys}
  The Euler--Lagrange system \eqref{eq:systeme} implies that
  \begin{equation}\label{eq:Reducesys}
    \begin{cases}
      \bH\3dots \bH_{0}-\bH_{0}\3dots \bH=0, \qquad & \text{(3 linear scalar equations)}    \\
      (\bH-\bH_{0})::\bH =0, \qquad                 & \text{(1 quadratic scalar equation)}  \\
      \bd'_{2}=0.  \qquad                           & \text{(5 quadratic scalar equations)}
    \end{cases}
  \end{equation}
\end{lem}

\begin{proof}
  By contracting three times the first equation $\bH-\bH_{0}+\bS(\blambda)=0$ in~\eqref{eq:systeme} with $\bH$ on the right
  and then on the left, we get
  \begin{subequations}
    \begin{align}
      \bH\3dots \bH-\bH_{0}\3dots \bH+\bS(\blambda)\3dots \bH & =0, \label{eq:24a}  \\
      \bH\3dots \bH-\bH\3dots \bH_{0}+\bH\3dots \bS(\blambda) & =0 . \label{eq:24b}
    \end{align}
  \end{subequations}
  By \eqref{eq:Sprim} and some calculations, we have
  \begin{equation*}
    \bS(\blambda)\3dots \bH = (\bH\cdot \blambda)^{s}\3dots \bH - \frac{3}{7} \left(\Id \odot (\bH:\blambda)\right)\3dots \bH = \frac{1}{4} \blambda\cdot \bd_{2} + \frac{3}{4} \bc-\frac{3}{14} \bH^2:\blambda
  \end{equation*}
  and
  \begin{equation*}
    \bH\3dots\bS(\blambda) = \bH \3dots (\bH\cdot \blambda)^{s} - \frac{3}{7} \bH \3dots \left(\Id \odot (\bH:\blambda)\right) = \frac{1}{4}  \bd_{2} \cdot \blambda + \frac{3}{4} \bc-\frac{3}{14} \bH^2:\blambda
  \end{equation*}
  where both $\bc$ and $\bH^{2}:\blambda=\bH:\bH:\blambda$ are symmetric second-order tensors with components
  \begin{equation*}
    c_{ij} = H_{ipqr}H_{jpqs} \lambda_{rs},
    \qquad
    (\bH^2:\blambda)_{ij} = H_{ijkl}H_{klmn}\lambda_{mn}.
  \end{equation*}
  If $\bH$ is at least cubic, then, $\bd_{2}'=0$ and $\blambda\cdot \bd_{2}=\bd_{2} \cdot \blambda$ is symmetric. Thus
  \begin{equation*}
    \bS(\blambda)\3dots \bH = \bH\3dots\bS(\blambda),
  \end{equation*}
  and, substracting \eqref{eq:24b} from \eqref{eq:24a}, we get
  \begin{equation*}
    \bH\3dots \bH_{0}-\bH_{0}\3dots \bH=0.
  \end{equation*}

  The second equation in~\eqref{eq:Reducesys} is obtained, by applying the Euler lemma on homogeneous functions to the quadratic function
  \begin{equation*}
    \bH \mapsto \frac{1}{2} \blambda : \bd_{2}'(\bH),
  \end{equation*}
  whose gradient is $\bS(\blambda)$. We get
  \begin{equation*}
    \bS(\blambda)::\bH = \frac{1}{2} \left(\grad_{\bH} (\blambda: \bd_{2}')\right)::\bH = \blambda: \bd_{2}'(\bH) = 0,
  \end{equation*}
  for each tensor $\bH$ which satisfies $\bd_{2}'(\bH)=0$. Therefore, contracting four times the first equation $\bH-\bH_{0}+\bS(\blambda)=0$ in~\eqref{eq:systeme} with $\bH$, we obtain the second equation of~\eqref{eq:Reducesys}, $(\bH-\bH_{0})::\bH=0$.
\end{proof}

\section{Numerical application -- Distance to cubic elasticity}
\label{sec:num-cubic-elasticity}

Let us now apply the Euler--Lagrange method to the problem of determining the distance
\begin{equation*}
  d(\bE_{0}, \text{cubic symmetry})=\min_{\bE \, \text{cubic}} \norm{\bE_{0}-\bE}=\norm{\bE_{0}-\bE^*},
\end{equation*}
of an experimental tensor $\bE_{0}$ to the cubic symmetry closed stratum. In our application the tensor $\bE_{0}$, taken from \cite{FGB1998} (refer to \cite{KRG1971,AHR1991,Art1993,FBG1996,Del2005,BAA2006} for measurements), is the elasticity tensor of a Nickel-based single crystal superalloy. In Voigt notation:
\begin{equation}\label{eq:E0}
  [\bE_{0}]=
  \begin{pmatrix}
    243 & 136 & 135 & 22 & 52 & -17 \\ 136 & 239 & 137 & -28 & 11 & 16 \\ 135 & 137 & 233 & 29 & -49 & 3 \\ 22 & -28 & 29 & 133 & -10 & -4 \\ 52 & 11 & -49 & -10 & 119 & -2 \\ -17 & 16 & 3 & -4 & -2 & 130
  \end{pmatrix} \;\text{ GPa},
  \qquad
  \norm{\bE_{0}}=713.41 \;\text{ GPa}.
\end{equation}
It can be checked (by \cite[Theorem 10.2 ]{OKDD2021}, see also \cite{FGB1998}) that the tensor $\bE_{0}$ is triclinic (with no material symmetry), even if it corresponds to a material with a so-called cubic $\gamma/\gamma'$ microstructure \cite{FS1987,PT2006,Ree2006}.

Using the formulas of theorem \ref{thm:Ecubic} we obtain the harmonic components of  $\bE_{0}$,
\begin{equation*}
  \lambda_{0}=\frac{1583}{15}\;  \text{ GPa},
  \qquad
  \mu_{0}=\frac{1453}{15}\;  \text{ GPa},
\end{equation*}
\begin{equation*}
  \bd_{0}'=\left(
  \begin{array}{ccc}
      \frac{11}{3} & 2           & 14            \\
      2            & \frac{5}{3} & 23            \\
      14           & 23          & -\frac{16}{3} \\
    \end{array}
  \right) \text{ GPa},
  \qquad
  \bv_{0}'=\left(
  \begin{array}{ccc}
      -1  & -11 & -1 \\
      -11 & 9   & -1 \\
      -1  & -1  & -8 \\
    \end{array}
  \right)
  \text{ GPa},
\end{equation*}
and, by \eqref{eq:H} (in Voigt notation),
\begin{equation}\label{eq:H0}
  [\bH_{0}]=\frac{1}{35}\left(
  \begin{array}{cccccc}
      -1986 & 1093  & 893   & 175   & 1760  & -495 \\
      1093  & -2306 & 1213  & -1085 & 15    & 660  \\
      893   & 1213  & -2106 & 910   & -1775 & -165 \\
      175   & -1085 & 910   & 1213  & -165  & 15   \\
      1760  & 15    & -1775 & -165  & 893   & 175  \\
      -495  & 660   & -165  & 15    & 175   & 1093 \\
    \end{array}
  \right)\text{ GPa}.
\end{equation}

The cost function $f=\norm{\bH_{0}-\bH}^{2}$ to minimize can then be expressed as (in GPa$^{2}$)
\begin{align*}
  f(\xx)= &
  8 \Lambda_{1}^{2}+2 \Lambda_{1} \Lambda_{2}+2 \Lambda_{1} \Lambda_{3}+668 \Lambda_{1}+8 \Lambda_{2}^{2}+2 \Lambda_{2} \Lambda_{3}+540 \Lambda_{2}+8 \Lambda_{3}^{2}+620 \Lambda_{3}
  \\
          & +16 X_{1}^{2}+8 X_{1} (X_{2}-11)+8 X_{2}^{2}-456 X_{2}+16 Y_{1}^{2}+8 Y_{1} Y_{2}-392 Y_{1}
  \\
          & +8 Y_{2}^{2}-808 Y_{2}+16 Z_{1}^{2}+8 Z_{1} Z_{2}-264 Z_{1}+8 Z_{2}^{2}-264 Z_{2}+ \frac{2026042}{35}
\end{align*}
in the variable
\begin{equation*}
  \xx=(X_{1}, X_{2}, Y_{1}, Y_{2}, Z_{1}, Z_{2}, \Lambda_{1}, \Lambda_{2}, \Lambda_{3}).
\end{equation*}
if the parameterization \eqref{eq:Hcubic} is used for $\bH$.

In terms of components, and according to the expression \eqref{eq:E0} for the considered material, the system of equations of lemma \ref{lem:sys} is constituted
\begin{itemize}
  \item[(a)]  \emph{of the three scalar equations,}
    \begin{align*}
      X_{1}= & \frac{1515991 \Lambda_{1}+6907074 \Lambda_{2}+2816520 \Lambda_{3}+4774213 Y_{2}+1319317 Z_{1}+3827136 Z_{2}}{2851559},
      \\
      X_{2}= & \frac{-2752251 \Lambda_{1}-5474665 \Lambda_{2}-1823999 \Lambda_{3}-3665127 Y_{2}+1198746 Z_{1}-1655027 Z_{2}}{2851559},
      \\
      Y_{1}= & \frac{-1401385 \Lambda_{1}-23691851 \Lambda_{2}-1939864 \Lambda_{3}-15828579 Y_{2}+4529623 Z_{1}+4531405 Z_{2}}{8554677},
    \end{align*}
    which correspond to the linear equation $\bH\3dots \bH_{0}-\bH_{0}\3dots \bH=0$,

  \item[(b)] \emph{of the scalar equation,}
    \begin{multline*}
      (\bH-\bH_{0})::\bH= 8\Lambda_1^2+2\Lambda_1\Lambda_2+2\Lambda_1\Lambda_3+8\Lambda_2^2+2\Lambda_2\Lambda_3+8\Lambda_3^2+16X_1^2+8X_1X_2+8X_2^2+16Y_1^2+8Y_1Y_2\\+8Y_2^2+16Z_1^2+8Z_1Z_2+8Z_2^2+334\Lambda_1+270\Lambda_2+310\Lambda_3-44X_1-228X_2-196Y_1-404Y_2-132Z_1-132Z_2
    \end{multline*}

  \item[(c)] \emph{and of the 5 equations} $g_{ij}=(\bd_{2})'_{ij}=0$, with
    \begin{align*}
      g_{11}=(\bd_{2}')_{11}= &
      \frac{2}{3} \big(-4 \Lambda_{1}^{2}-\Lambda_{1} \Lambda_{2}-\Lambda_{1} \Lambda_{3}+2 \Lambda_{2}^{2}+2 \Lambda_{2} \Lambda_{3}+2 \Lambda_{3}^{2}+X_{1}^{2}-4 X_{1} X_{2}-4 X_{2}^{2}
      \\
      \nonumber
                              & \qquad +Y_{1}^{2}+5 Y_{1} Y_{2}+2 Y_{2}^{2}-2 Z_{1}^{2}-Z_{1} Z_{2}+2 Z_{2}^{2}\big),
      \\
      g_{22}=(\bd_{2}')_{22}= & -\frac{2}{3} \big(-2 \Lambda_{1}^{2}+\Lambda_{1} \Lambda_{2}-2 \Lambda_{1} \Lambda_{3}+4 \Lambda_{2}^{2}+\Lambda_{2} \Lambda_{3}-2 \Lambda_{3}^{2}+2 X_{1}^{2}+X_{1} X_{2}-2 X_{2}^{2}
      \\
      \nonumber
                              & \qquad -Y_{1}^{2}+4 Y_{1} Y_{2}+4 Y_{2}^{2}-Z_{1}^{2}-5 Z_{1} Z_{2}-2 Z_{2}^{2}\big),
      \\
      g_{12}=(\bd_{2}')_{12}= & 3 X_{1} Y_{1} + 3 X_{2} Y_{1} - 4 X_{1} Y_{2} - X_{2} Y_{2} + 4 Z_{1}\Lambda_{1} +
      Z_{2}\Lambda_{1} + 3 Z_{1}\Lambda_{2} -
      Z_{2}\Lambda_{2} - 2 Z_{1}\Lambda_{3},
      \\
      g_{13}=(\bd_{2}')_{13}= & 3 X_{1} (Z_{1} + Z_{2}) - X_{2} (4 Z_{1} + Z_{2}) + 3 Y_{1}\Lambda_{1} -
      Y_{2}\Lambda_{1} - 2 Y_{1}\Lambda_{2} +
      4 Y_{1}\Lambda_{3} + Y_{2}\Lambda_{3},
      \\
      g_{23}=(\bd_{2}')_{23}= & 3 Y_{1} Z_{1} + 3 Y_{2} Z_{1} - 4 Y_{1} Z_{2} - Y_{2} Z_{2} - 2 X_{1}\Lambda_{1} +
      4 X_{1}\Lambda_{2} + X_{2}\Lambda_{2} +
      3 X_{1}\Lambda_{3} - X_{2}\Lambda_{3}.
    \end{align*}
\end{itemize}

Using the first three linear equations (of point (a)), we further reduce the system to 6 equations
\begin{equation*}
  g_{ij}=0
  \quad
  \text{and}
  \quad
  g_6=(\bH-\bH_{0})::\bH= 0,
\end{equation*}
quadratic in the 6 variables $Y_{2}, Z_{1}, Z_{2}, \Lambda_{1}, \Lambda_{2}, \Lambda_{3}$, and which can be solved thanks to the determination of a Gröbner basis $\mathrm{GB}$, by symbolic computation using Mathematica software\footnote{by the command $\mathrm{GB}=\mathbf{GroebnerBasis}\left[\set{g11, g22, g12, g13, g23, g6}, \set{Y2, Z1, Z2, \Lambda{1}, \Lambda{2}, \Lambda{3}} \right]$, where by default the lexicographic elimination order is used.}. We take advantage of the fact that the material parameters (here the components of $\bE_{0}$), are measured with only a few significant digits to work with rational coefficients polynomials. This point is of main importance in the resolution of a system of polynomial equations by the obtention of a Gröbner basis (see remark \ref{rem:coeffQ} of the Appendix).
The result is a set $\mathrm{GB}=\set{\mathrm{GB}_{1},\dotsc , \mathrm{GB}_{32} }$ of $32$ polynomials $\mathrm{GB}_n$ (unfortunately too lengthy to be given) in the variables $Y_{2}, Z_{1}, Z_{2}, \Lambda_{1}, \Lambda_{2}, \Lambda_{3}$, and which vanishes if and only if the initial (polynomial) system \eqref{eq:Reducesys} is satisfied.

In the present application, the first polynomial of the Gröbner basis $\mathrm{GB}_{1}$ is found to be function of $\Lambda_{3}$ only, $\mathrm{GB}_{2}$ function of $\Lambda_{2}$ and $\Lambda_{3}$ (but linear in $\Lambda_{2}$), and so on, up to  $\mathrm{GB}_{32}$ function of all the variables (but linear in $Y_{2}$), as in~\eqref{eq:GB-system-1} of \autoref{sec:groebner} with $n=6$ and $x_{6}=\Lambda_{3}$. Solving $\mathrm{GB}_{1}(\Lambda_{3})=0$ (using the command $\mathbf{NSolve}[\mathrm{GB}[[1]] == 0, \Lambda{3}, \mathbf{WorkingPrecision} \rightarrow 50]$), we get either $\Lambda_{3}=\Lambda_{3}^{(0)}=0$ (leading to the isotropic solution $\bH^{(0)}=0$) or $\Lambda_{3}$ is a real root of a polynomial of degree 14, which has 8 non-zero real roots (in practice determined with a 50 significant digits precision),
\begin{align*}
  \Lambda_{3}^{(1)} & = -38.908854, & \Lambda_{3}^{(2)} & = -10.425971, & \Lambda_{3}^{(3)} & = -8.424314 , & \Lambda_{3}^{(4)} & = -6.225368,
  \\
  \Lambda_{3}^{(5)} & = -3.194952 , & \Lambda_{3}^{(6)} & = -3.056232,  & \Lambda_{3}^{(7)} & = 1.745698,   & \Lambda_{3}^{(8)} & = 13.541284.
\end{align*}

Except from this initial (roots) solving, the remaining unknowns $\Lambda_{2}$, then $\Lambda_{1}, Z_{2}, Z_{1}$ and last $Y_{2}$, are obtained analytically one per one for each $\Lambda_{3}^{(s)}$ solution (thanks to the equations $\mathrm{GB}_m=0$, $m\geq 2$, given by the elements of the Gröbner basis GB, when $\Lambda_{3}$ is evaluated). The variables $X_{1}$, $X_{2}$, $Y_{1}$ are finally given by the three linear equations of point (a).

This polynomial optimization approach shows that, generically, for the distance to cubic symmetry problem, the number of critical points solutions of the first-order Euler--Lagrange equations~\eqref{eq:Reducesys} is finite, the corresponding solutions $\bH^{(s)}$ being fully determined by all the roots of the polynomials in the Gröbner basis GB. The global minimum $\min f(\bH)$ is simply the minimum minimorum
\begin{equation*}
  \min_{1 \le s \le 8} \norm{\bH_{0}-\bH^{(s)}}^{2} = \norm{\bH_{0}-\bH^{(1)}}^{2}=2530.47\, \mathrm{GPa}^2,
\end{equation*}
which is here given by the solution $s=1$, $\Lambda_{3}=\Lambda_{3}^{(1)}$,
\begin{align*}
  X_{1}       & = -6.396655,  & X_{2}       & = 27.780761,  & Y_{1}       & =-2.277535,
  \\
  Y_{2}       & = 44.251233,  & Z_{1}       & = -4.557361,  & Z_{2}       & =21.161420,
  \\
  \Lambda_{1} & = -36.401302, & \Lambda_{2} & = -20.226895, & \Lambda_{3} & = -38.908854,
\end{align*}
for $\bH^*$. The numerical value $f(0)=\norm{\bH_{0}}^2=57886.9\, \mathrm{GPa}^2$ for $\bH$ isotropic is found larger than
the one $2530.47\, \mathrm{GPa}^2$ for the optimal cubic tensor $\bH$.

With the values $\lambda=\lambda_{0}={1583}/{15}=105.533333$ and $\mu=\mu_{0}={1453}/{15}=96.866667$,
the tensor
\begin{equation*}
  \bE^*=2\mu_{0} \bI +\lambda_{0} \Id \otimes \Id+ \bH^*,
\end{equation*}
of Voigt representation
\begin{equation}\label{eq:Ecubic}
  [\bE^*]
  ={\scriptsize\left(
  \begin{array}{cccccc}
      240.130916  & 144.442188  & 125.760229  & 6.39665526  & 41.9736976  & -21.1614201 \\
      144.442188  & 223.956510  & 141.934636  & -27.7807617 & 2.27753546  & 16.6040582  \\
      125.760229  & 141.934636  & 242.638469  & 21.3841064  & -44.2512331 & 4.55736193  \\
      6.39665526  & -27.7807617 & 21.3841064  & 133.267969  & 4.55736193  & 2.27753546  \\
      41.9736976  & 2.27753546  & -44.2512331 & 4.55736193  & 117.093562  & 6.39665526  \\
      -21.1614201 & 16.6040582  & 4.55736193  & 2.27753546  & 6.39665526  & 135.775521  \\
    \end{array}
  \right)
  }
  \text{ GPa}
\end{equation}
is the (cubic) elasticity tensor that minimizes the distance to cubic symmetry, with then
\begin{equation*}
  d(\bE_{0}, \text{cubic symmetry})=74.13
  \, \text{ GPa}.
\end{equation*}
With a relative distance
\begin{equation*}
  \frac{\norm{\bE_{0}-\bE^{*}}}{ \norm{\bE_{0}}}=0.1039,
\end{equation*}
it is slightly better than the solution obtained by François--Geymonat--Berthaud by a numerical iterative method~\cite{FGB1998}.

As $\bH^*\neq 0$, the tensor $\bE^{*}$ is cubic. The distance of $\bE_{0}$ to isotropy,
\begin{equation*}
  d(\bE_{0}, \text{isotropy})=
  \norm{\bE_{0}-\left(2\mu_{0} \bI +\lambda_{0} \Id \otimes \Id\right)}=246.68
  \, \text{ GPa},
\end{equation*}
is found larger than the one to cubic symmetry, with a relative distance to isotropy
\begin{equation*}
  \frac{d(\bE_{0}, \text{isotropy})}{\norm{\bE_{0}}}=0.3458.
\end{equation*}

By remark \ref{rem:normal-form-hooke}, the normal form (denoted here by $\bE^{*}_\octa$) of the optimal cubic elasticity tensor
$\bE^{*}=(\lambda=\lambda_{0}, \mu=\mu_{0}, 0, 0, \bH^*)$ given by \eqref{eq:Ecubic}, is obtained directly
thanks to the computation of its invariants. We get, by the explicit formulas \eqref{eq:delta} to \eqref{eq:J3detail},
\begin{equation*}
  J_{2}=\norm{\bH^*}^{2}=55356.440
  \, \mathrm{GPa}^{2},
  \qquad
  J_{3}=\tr_{13} (\bH^{*\,3})=-2377889.1
  \, \mathrm{GPa}^{3},
\end{equation*}
so that, in GPa,
\begin{equation*}
  \lambda=105.533333
  , \qquad
  \mu=96.866667
  ,
  \qquad
  \delta=\frac{J_{3}}{4J_{2}}=-10.738990,
\end{equation*}
and
\begin{equation}\label{eq:EcubicNormal}
  [\bE^*_{\octa}]={\footnotesize
  \begin{pmatrix}
    213.354743
      &
    148.489295
      & 148.489295
      & 0          & 0          & 0                                    \\
    148.489295
      & 213.354743
      & 148.489295
      & 0          & 0          & 0                                    \\
    148.489295
      & 148.489295 & 213.354743
      & 0          & 0          & 0                                    \\
    0 & 0          & 0          & 139.822628
      & 0          & 0                                                 \\
    0 & 0          & 0          & 0          & 139.822628
      & 0                                                              \\
    0 & 0          & 0          & 0          & 0          & 139.822628
  \end{pmatrix}_{(\ee_{1}, \ee_{2}, \ee_{3})}}
  \ \text{GPa},
\end{equation}
which, for practical applications, can be by approximated by
\begin{equation}
  [\bE^*_{\octa}]=
  \begin{pmatrix}
    213
      &
    148.5
      & 148.5
      & 0     & 0   & 0               \\
    148.5
      & 213
      & 148.5
      & 0     & 0   & 0               \\
    148.5
      & 148.5 & 213
      & 0     & 0   & 0               \\
    0 & 0     & 0   & 140
      & 0     & 0                     \\
    0 & 0     & 0   & 0   & 140
      & 0                             \\
    0 & 0     & 0   & 0   & 0   & 140
  \end{pmatrix}_{(\ee_{1}, \ee_{2}, \ee_{3})}
  \ \text{GPa}.
\end{equation}

\section{Distance to cubic elasto-plasticity as a polynomial optimization problem}
\label{sec:distance-cubic-elasto-plasticity}

The anisotropic Hill elasto-plasticity theory for metallic materials introduces not one but two fourth-order constitutive tensors \cite{Hil1948,LC1985,Hil1998,BCCF2012},
\begin{itemize}
  \item a first one, $\bE\in \Ela$, to describe the anisotropic elasticity,
  \item a second one, $\bP$ (sometimes considered as dimensionless), to describe the yield (plasticity)  criterion, and such as
        the condition
        \begin{equation*}
          \bsigma':\bP: \bsigma' - R^{2}<0
        \end{equation*}
        corresponds to an elastic loading or unloading stage (with $\bsigma'\in \HH^{2}$ the continuum mechanics deviatoric stress tensor). When assumed constant, the scalar $R$ stands for the material yield stress, when taken as evolving during loading, it stands for the material hardening. The Hill tensor $\bP$ has the indicial symmetries of elasticity tensors (so that $\bP\in \Ela$).
\end{itemize}

With no lack of generality, instead of $\bP$, we can work with a tensor $\bF$ of elasticity-type, and compute a dimensionless Hill tensor $\bP=\bF/C$ by normalizing afterward $\bF$ with a constant $C$. Indeed, when $\bF$ is in its normal form~\eqref{eq:VoigtCubic}, setting
\begin{equation}\label{eq:C-L}
  C:=\frac{2 }{3}(F_{1111}-F_{1122}) \quad \text{and} \quad
  L:=\frac{3 F_{1212}}{F_{1111}-F_{1122}},
\end{equation}
allows to recover the standard expression of cubic Hill yield criterion (in cubic basis $(\ee_{i})$), as
\begin{equation*}
  \bsigma':\bP:\bsigma'
  =\frac{1}{2} \left((\sigma_{11}-\sigma_{22})^{2}+(\sigma_{33}-\sigma_{11})^{2}+(\sigma_{22}-\sigma_{33})^{2}\right)+2 L \left(\sigma_{12}^{2}+\sigma_{13}^{2}+\sigma_{23}^{2}\right).
\end{equation*}

The harmonic decomposition of $\bF$ is then (see remark \ref{rem:HrmDecompNorm2})
\begin{equation*}
  \bF=\left(\ell, m , \be', \bw', \bK\right),
\end{equation*}
with
\begin{equation*}
  \ell=\frac{1}{15}(2 \tr\be-\tr\bw) \quad \text{and} \quad m=\frac{1}{30}(3\tr\bw - \tr\be),
\end{equation*}
the Lamé constants of $\bF$, where
\begin{equation*}
  \be:=\tr_{12}\bF \quad \text{and} \quad \bw:=\tr_{13}\bF ,
\end{equation*}
are respectively the dilatation and Voigt tensors of $\bF$, and  $\bK=(\bF)^{s\prime}\in \HH^{4}$ is the harmonic fourth-order component of $\bF$, given by \eqref{eq:H},
\begin{equation*}
  \bK=(\bF)^{s} -(2m+ \ell) \Id \odot \Id -
  \frac{2}{7}\Id \odot ( \be^{\prime} +2  \bw^{\prime}).
\end{equation*}

We now assume that two given elasto-plasticity tensors $\bE_{0}$ and $\bF_{0}$ are available (possibly triclinic) for a given metallic material. As a generalization of the formulation of the distance problem of \autoref{sec:distance-cubic-elasticity}, in which only one constitutive tensor (the elasticity tensor) was involved, we propose to define the optimum cubic estimates $\bE^{**}$ and $\bF^{**}=C \bP^{**}$ of the two elasto-plasticity constitutive tensors, as the minimizers of the following quadratic function (with $W$ a given strictly positive weight)
\begin{equation*}
  f(\bE, \bF):= \norm{\bE_{0}-\bE}^{2}+W \norm{\bF_{0}-\bF}^{2},
\end{equation*}
at given  tensors $\bE_{0}$ and $\bF_{0}=C \bP_{0}$, under the constraint that both the elasticity tensor $\bE$ and the Hill tensor $\bF=C\bP$ are cubic and share the same cubic axes (by theorem \ref{thm:EFcubic}). The introduction of a weight $W$ is necessary in practice, since the orders of magnitude (and the units) of the Hooke and Hill
tensors are often very different.

\begin{rem}\label{rem:normElaEla}
  $\sqrt{\norm{\bE}^{2}+W \norm{\bF}^{2}}$, with $W>0$, is a norm on $\Ela \oplus \Ela$.
\end{rem}

We first perform the harmonic decompositions of $\bE_{0}$ and $\bF_{0}$,
\begin{equation*}
  \bE_{0}=\left(\lambda_{0}, \mu_{0} , \bd_{0}', \bv_{0}', \bH_{0}\right),
  \qquad
  \bF_{0}=\left(\ell_{0}, m_{0} , \be_{0}', \bw_{0}', \bK_{0}\right),
\end{equation*}
with $\lambda_{0}, \mu_{0}, \ell_{0}, m_{0}\in \HH^0$, $\bd_{0}', \bv_{0}',\be_{0}', \bw_{0}'\in \HH^2$ and
$\bH_{0}, \bK_0 \in \HH^{4}$ their harmonic components.
The harmonic decompositions of
the sought cubic tensors $\bE$ and $\bF$ are
\begin{equation*}
  \bE=\left(\lambda, \mu , 0, 0, \bH\right),
  \qquad
  \bF=\left(\ell, m , 0, 0 , \bK\right),
\end{equation*}
with $\lambda, \mu, \ell, m\in \HH^0$,
$\bH, \bK \in \HH^{4}$,
and, according to \eqref{eq:Eharm}, we have
\begin{equation*}
  \bE=2\mu \bI + \lambda \Id \otimes \Id+\bH
  \quad \text{and}\quad
  \bF=2m \bI + \ell \Id \otimes \Id+\bK,
\end{equation*}
with $\bd_{2}'(\bH)=\bd_{2}'(\bK)=0$ and $\bK=k \bH$ (by theorem \ref{thm:Ecubic}). Using the formula~\eqref{eq:normE2} for both $\norm{\bE_{0}-\bE}^{2}$ and $\norm{\bF_{0}-\bF}^{2}$, we get
\begin{align*}
  f(\bE) & =3 \left(3 (\lambda_{0}-\lambda)^{2} + 4 (\lambda_{0}-\lambda) (\mu_{0}-\mu) +8 (\mu_{0}-\mu) ^{2}\right)
  \\
         & \quad+ 3W \left(3 (\ell_{0}-\ell)^{2} + 4 (\ell_{0}-\ell) (m_{0}-m) +8 (m_{0}-m) ^{2}\right)
  \\
         & \quad + \frac{2}{21} \norm{\bd_{0}^{\prime} + 2 \bv_{0}^{\prime}}^{2} + \frac{4}{3}\norm{\bd_{0}^{\prime} - \bv_{0}^{\prime}}^{2} + \frac{2W}{21} \norm{\be_{0}^{\prime} + 2 \bw_{0}^{\prime}}^{2} + \frac{4W}{3}\norm{\be_{0}^{\prime} - \bw_{0}^{\prime}}^{2}
  \\
         & \quad +   \norm{\bH_{0}-\bH}^{2} + W\norm{\bK_{0}-k \bH}^{2}.
\end{align*}
The minimum of this expression is obtained for
\begin{equation*}
  \lambda=\lambda_{0}, \quad \mu=\mu_{0}, \quad \ell=\ell_{0}, \quad m=m_{0}, \quad \bH = \bH^{*}, \quad k = k^{*},
\end{equation*}
where $\bH^{*}$ and $k^{*}$ correspond to absolute minima of the problem
\begin{equation*}
  \min_{\bH,k} \left\{\norm{\bH_{0}-\bH}^{2}+W \norm{\bK_{0}-k \bH}^{2}\right\}, \quad \text{with} \quad \bd_{2}'(\bH)=(\bH\3dots \bH)'=0.
\end{equation*}

\begin{rem}
  Note that the condition $\bK = k\bH\ne 0$ implies that the pair $(\bH,\bK)$ is cubic, meaning that both $\bH$ and $\bK$ are cubic and share the same cubic axes.
\end{rem}

To solve the problem of the distance of a pair $(\bE_0, \bF_0)$ to cubic symmetry, we therefore have to find the critical points of the polynomial function
\begin{equation}\label{eq:FEP}
  F(\bH, k, \blambda):=\norm{\bH_{0}-\bH}^{2}+W \norm{\bK_{0}-k \bH}^{2}
  +\blambda:g,
\end{equation}
with $\bH\in \HH^{4}$ an harmonic fourth-order tensor, $k$ a scalar, and where the Lagrange multiplier $\blambda\in \HH^{2}$ is an deviatoric second-order tensor. Observe that the first-order Euler--Lagrange equations for this optimization problem can furthermore be recast in a similar form as \eqref{eq:Reducesys}.

\begin{lem}\label{lem:sysEF}
  The first-order Euler--Lagrange equations,
  \begin{equation*}
    \frac{\partial F}{\partial \bH}=0, \quad \frac{\partial F}{\partial k}=0
    \quad
    \text{and} \quad \frac{\partial F}{\partial \blambda}=0,
  \end{equation*}
  imply
  \begin{equation}\label{eq:pbHK}
    \begin{cases}
      \bH \3dots (\bH_{0}+kW \bK_{0}) -(\bH_{0}+kW \bK_{0})\3dots \bH=0 \qquad & \text{(3 scalar equations)} \\
      (\bH-\bH_{0})::\bH =0 \qquad                                             & \text{(1 scalar equation)}  \\
      W (k \bH-\bK_{0})::\bH =0 \qquad                                         & \text{(1 scalar equation)}  \\
      \bd'_{2}=0  \qquad                                                       & \text{(5 scalar equations)}
    \end{cases}
  \end{equation}
\end{lem}

\begin{rem}
  The distance problem thus  formulated is not a quadratic optimization problem. The equation
  \begin{equation*}
    (k \bH-\bK_{0})::\bH =0
  \end{equation*}
  is indeed polynomial, but of degree three in the variable $\xx=(\bH, k)$.
\end{rem}

The first equation of \eqref{eq:pbHK} is not linear anymore, it cannot be used to reduce the number of unknowns before the computation of a Gröbner basis. The quasi-analytical resolution by the obtention of a  Gröbner basis will nevertheless be similar (but with four more variables) to the resolution for the single elasticity tensor case (except that the computation of a Gröbner basis  will be more computer time consuming).

\section{Numerical application -- Distance to cubic elasto-plasticity}
\label{sec:num-cubic-elasto-plasticity}

We consider here the example of the triclinic elasticity tensor $\bE_{0}$ (still given by \eqref{eq:E0}, the harmonic decomposition $\bE_0=\left(\lambda_{0}, \mu_{0} , \bd_{0}', \bw', \bH_{0}\right)$ remaining the one of \autoref{sec:num-cubic-elasticity}), and of the following triclinic  plasticity tensor $\bF_{0}$, in Voigt notation,
\begin{equation*}\label{eq:F0}
  [\bF_{0}]=
  \left(
  \begin{array}{cccccc}
      191 & -54 & -83 & -34 & -94 & 59  \\
      -54 & 176 & -71 & 71  & -40 & -23 \\
      -83 & -71 & 207 & -44 & 130 & -36 \\
      -34 & 71  & -44 & 99  & -15 & -17 \\
      -94 & -40 & 130 & -15 & 179 & -40 \\
      59  & -23 & -36 & -17 & -40 & 79  \\
    \end{array}
  \right)
  ,
  \qquad
  \norm{\bF_{0}}=715.78.
\end{equation*}
Using the formulas of theorem \ref{thm:Ecubic}, we obtain $\bF_{0}=\left(\ell_{0}, m_{0} , \be_{0}', \bw_{0}', \bK_{0}\right)$, with
\begin{equation*}
  \ell_{0}=-\frac{324}{5},
  \qquad
  m_{0}=\frac{1853}{15},
\end{equation*}
\begin{equation*}
  \be_{0}'=(\tr_{12}\bF_{0})'=\left(
  \begin{array}{ccc}
      \frac{4}{3} & 0            & -4          \\
      0           & -\frac{5}{3} & -7          \\
      -4          & -7           & \frac{1}{3} \\
    \end{array}
  \right),
  \quad
  \bw_{0}'=(\tr_{13}\bF_{0})'=\left(
  \begin{array}{ccc}
      \frac{59}{3} & 21             & 19            \\
      21           & -\frac{226}{3} & -13           \\
      19           & -13            & \frac{167}{3} \\
    \end{array}
  \right)
  .
\end{equation*}
and (in Voigt notation)
\begin{equation*}
  [\bK_{0}]=\frac{1}{35}
  \left(
  \begin{array}{cccccc}
      -101  & -727  & 828   & -1275 & -3460 & 1855  \\
      -727  & 1304  & -577  & 2650  & -920  & -1015 \\
      828   & -577  & -251  & -1375 & 4380  & -840  \\
      -1275 & 2650  & -1375 & -577  & -840  & -920  \\
      -3460 & -920  & 4380  & -840  & 828   & -1275 \\
      1855  & -1015 & -840  & -920  & -1275 & -727  \\
    \end{array}
  \right).
\end{equation*}

The cost function
\begin{equation*}
  f(\bH)=\norm{\bH_{0}-\bH}^{2}+ W \norm{\bK_{0}-k \bH}^{2},
\end{equation*}
that we have to minimize in order to solve the distance problem
\begin{equation*}
  \min_{(\bE, \bF)\, \text{cubic}} \left(
  \norm{\bE_{0}-\bE}^{2}+ W \norm{\bF_{0}-\bF}^{2}
  \right),
\end{equation*}
is (in GPa$^{2}$)
\begin{align*}
  f(\xx) & =
  8 \Lambda_{1}^{2}+2 \Lambda_{1} \Lambda_{2}+2 \Lambda_{1} \Lambda_{3}+668 \Lambda_{1}+8 \Lambda_{2}^{2}+2 \Lambda_{2} \Lambda_{3}+540 \Lambda_{2}+8 \Lambda_{3}^{2}+620 \Lambda_{3}
  \\
         & \quad +16 X_{1}^{2}+8 X_{1} (X_{2}-11)+8 X_{2}^{2}-456 X_{2}+16 Y_{1}^{2}+8 Y_{1} Y_{2}-392 Y_{1}
  \\
         & \quad +8 Y_{2}^{2}-808 Y_{2}+16 Z_{1}^{2}+8 Z_{1} Z_{2}-264 Z_{1}+8 Z_{2}^{2}-264 Z_{2}+ \frac{2026042}{35}
  \\
         & \quad +W \Big(16 X_{1}^{2} k^{2}+8 X_{1} k (X_{2} k-70)+8 X_{2}^{2} k^{2}+920 X_{2} k+16 Y_{1}^{2} k^{2}+8 Y_{1} Y_{2} k^{2}+160 Y_{1} k
  \\
         & \quad +8 Y_{2}^{2} k^{2}+1792 Y_{2} k+8 \Lambda_{1}^{2} k^{2}+2 \Lambda_{1} \Lambda_{2} k^{2}+2 \Lambda_{1} \Lambda_{3} k^{2}+8 \Lambda_{2}^{2} k^{2}+2 \Lambda_{2} \Lambda_{3} k^{2}+8 \Lambda_{3}^{2} k^{2}+16 k^{2} Z_{1}^{2}
  \\
         & \quad +8 k^{2} Z_{1} Z_{2}+8 k^{2} Z_{2}^{2}-258 \Lambda_{1} k+304 \Lambda_{2} k-318 \Lambda_{3} k-344 k Z_{1}+656 k Z_{2}+\frac{6495682}{35}\Big) .
\end{align*}
It is expressed in the variable
\begin{equation*}
  \xx = (k, X_{1}, X_{2}, Y_{1}, Y_{2}, Z_{1}, Z_{2}, \Lambda_{1}, \Lambda_{2}, \Lambda_{3}),
\end{equation*}
if the parameterization \eqref{eq:Hcubic} is used for $\bH$. The first-order Euler-Lagrange equations are given in lemma \ref{lem:sysEF}. In components, they consist of
\begin{itemize}
  \item[(a)]  \emph{the three scalar equations,}
    \begin{align*}
       & (a1) &  & -43 k \Lambda _{1} W+125 k \Lambda _{2} W+328 k \Lambda _{3} W+X_{1} (199-264 k W)+X_{2} (52-204 k W)
      \\
       &      &  & -555 k W Y_{1}-70 k W Y_{2}+220 k W Z_{1}+159 k W Z_{2}-33 \Lambda _{1}-132 \Lambda _{3}+138 Y_{1}
      \\
       &      &  & -11 Y_{2}-187 Z_{1}-310 Z_{2}=0,
      \\
       & (a2) &  & k W (-204 \Lambda _{1}-896 \Lambda _{2}-20 \Lambda _{3}+375 X_{1}+43 X_{2}+61 Y_{1}+152 Y_{2}-25 Z_{1}+185 Z_{2}) \\
       &      &  & +52 \Lambda _{1}+404 \Lambda _{2}+49 \Lambda _{3}+33 X_{2}+123 Y_{1}
      +270 Y_{2}-79 Z_{1}-46 Z_{2}=0,
      \\
       & (a3) &  & 460 k \Lambda _{1} W-70 k \Lambda _{2} W+185 k \Lambda _{3} W-7 X_{1} (13 k W+21)+X_{2} (129 k W-334)
      \\
       &      &  & +4 k W Y_{1}-125 k W Y_{2}-612 k W Z_{1}+20 k W Z_{2}-228 \Lambda _{1}-11 \Lambda _{2}-46 \Lambda _{3}
      \\
       &      &  & +99 Y_{1}+156 Z_{1}-49 Z_{2}=0,
    \end{align*}
    which correspond to the linear equation $  \bH \3dots (\bH_{0}+kW \bK_{0}) -(\bH_{0}+kW \bK_{0})\3dots \bH=0$,

  \item[(b)] \emph{the scalar equation} $(\bH-\bH_{0})::\bH=0$ (detailed in point (b) of \autoref{sec:num-cubic-elasticity}),

  \item[(c)] \emph{the scalar equation}
    \begin{align*}
      (k \bH-\bK_{0})::\bH & =
      8 k \Lambda _{1}^{2}+2 k \Lambda _{1} \Lambda _{2}+2 k \Lambda _{1} \Lambda _{3}+8 k \Lambda _{2}^{2}+2 k \Lambda _{2} \Lambda _{3}+8 k \Lambda _{3}^{2}+16 k X_{1}^{2}
      \\
                           & \qquad +8 X_{1} (k X_{2}-35)+8 k X_{2}^{2}+16 k Y_{1}^{2}+8 k Y_{1} Y_{2}+8 k Y_{2}^{2}+16 k Z_{1}^{2}
      \\
                           & \qquad+8 k Z_{1} Z_{2}+8 k Z_{2}^{2}-129 \Lambda _{1}+152 \Lambda _{2}-159 \Lambda _{3}+460 X_{2}+80 Y_{1}
      \\
                           & \qquad +896 Y_{2}-172 Z_{1}+328 Z_{2}
      \\
                           & =0
      ,
    \end{align*}

  \item[(d)] \emph{and the 5 equations} $g_{ij}=(\bd_{2})'_{ij}=0$ (detailed in point (c) of \autoref{sec:num-cubic-elasticity}).
\end{itemize}

We set a unit weight $W=1$ for the numerical application. The resolution is similar to the one for the single elasticity tensor case, except that now the variable $\xx$ is 10-dimensional, and that there is no \emph{a priori} reduction in the number of scalar unknowns. Rational coefficients are considered for the given tensors $\bE_0$ and $\bF_0$ (and for their harmonic components $\bH_0$ and $\bK_0$). A Gröbner basis $\mathrm{GB}=\set{\mathrm{GB}_{1}, \dotsc , \mathrm{GB}_{111}}$ of 111 elements is computed using Mathematica. Its first element $\mathrm{GB}_{1}$ is found to be a polynomial in $\Lambda_{3}$ only; $\Lambda_{3}$ is either zero (leading to the isotropic solution $\bH=0$) or it is a solution of a polynomial equation of degree 56, which has 18 real non zero roots (in practice determined with a 100 significant digits precision). Once $\mathrm{GB}_{1}(\Lambda_{3})=0$ is solved, the remaining Gröbner basis equations are linear (as in~\eqref{eq:GB-system-1} of \autoref{sec:groebner}) in the variables $\Lambda_{2}$, $\Lambda_{1}$, \ldots , $X_{2}$, $X_{1}$, and $k$.

The minimum minimorum for the cost function is here given by the solution $\Lambda_{3}=-19.612165$
(it is not given by the isotropic solution $\bH=\bK=0$). We get the optimal value $k^{**}=-2.134021$ for $k$ and (in GPa):
\begin{align*}
  X_{1}       & =-16.788457,  & X_{2}       & = 39.191663, & Y_{1}       & = -8.812379,
  \\
  Y_{2}       & = 43.001809,  & Z_{1}       & = -8.048394, & Z_{2}       & = 30.315189,
  \\
  \Lambda_{1} & = -15.769513, & \Lambda_{2} & = 5.950665,  & \Lambda_{3} & = -19.612165,
\end{align*}
so that the optimal tensor $\bH^{**}$ has expression (in Voigt notation)
\begin{equation*}
  [\bH^{**}]
  ={\scriptsize
  \left(
  \begin{array}{cccccc}
      -13.661500 & 19.612165  & -5.950665  & 16.788457  & 34.189430  & -30.315189 \\
      19.612165  & -35.381678 & 15.769513  & -39.191663 & 8.812379   & 22.266795  \\
      -5.950665  & 15.769513  & -9.818848  & 22.403206  & -43.001809 & 8.048394   \\
      16.788457  & -39.191663 & 22.403206  & 15.769513  & 8.048394   & 8.812379   \\
      34.189430  & 8.812379   & -43.001809 & 8.048394   & -5.950665  & 16.788457  \\
      -30.315189 & 22.266795  & 8.048394   & 8.812379   & 16.788457  & 19.612165  \\
    \end{array}
  \right)}
  \text{ GPa}.
\end{equation*}
With the values $\lambda=\lambda_{0}=105.533333$ and $\mu=\mu_{0}=96.866667$,
the optimal cubic elasticity tensor
\begin{equation*}
  \bE^{**}=2\mu_{0} \bI +\lambda_{0} \Id \otimes \Id+ \bH^{**},
\end{equation*}
has Voigt representation,
\begin{equation*}\label{eq:Estarstarccubic}
  [\bE^{**}]
  ={\scriptsize
  \left(
  \begin{array}{cccccc}
      285.605167 & 125.145498 & 99.582668  & 16.788457  & 34.189430  & -30.315189 \\
      125.145498 & 263.884989 & 121.302846 & -39.191663 & 8.812379   & 22.266795  \\
      99.582668  & 121.302846 & 289.447819 & 22.403206  & -43.001809 & 8.048394   \\
      16.788457  & -39.191663 & 22.403206  & 112.636180 & 8.048394   & 8.812379   \\
      34.189430  & 8.812379   & -43.001809 & 8.048394   & 90.916002  & 16.788457  \\
      -30.315189 & 22.266795  & 8.048394   & 8.812379   & 16.788457  & 116.478831 \\
    \end{array}
  \right)}
  \text{ GPa}.
\end{equation*}
Since $\ell=\ell_{0}=-64.800000$ and $m=m_{0}=123.533333$, we get for the optimal cubic plasticity tensor
\begin{equation*}
  \bF^{**}=2m_{0} \bI +\ell_{0} \Id \otimes \Id+ k^{**} \bH^{**},
\end{equation*}
the Voigt representation
\begin{equation*}\label{eq:Fstarstarcubic}
  [\bF^{**}]={\scriptsize
  \left(
  \begin{array}{cccccc}
      211.420595  & -106.652774 & -52.101155 & -35.826922 & -72.960965 & 64.693254  \\
      -106.652774 & 257.771914  & -98.452474 & 83.635836  & -18.805804 & -47.517810 \\
      -52.101155  & -98.452474  & 203.220295 & -47.808914 & 91.766769  & -17.175444 \\
      -35.826922  & 83.635836   & -47.808914 & 89.880860  & -17.175444 & -18.805804 \\
      -72.960965  & -18.805804  & 91.766769  & -17.175444 & 136.232178 & -35.826922 \\
      64.693254   & -47.517810  & -17.175444 & -18.805804 & -35.826922 & 81.680560  \\
    \end{array}
  \right)}
  .
\end{equation*}
The relative distance to cubic symmetry for this two constitutive elasto-plasticity tensors problem is
\begin{equation*}
  \sqrt{\frac{\norm{\bE_{0}-\bE^{**}}^{2}+ \norm{\bF_{0}-\bF^{**}}^{2}}{\norm{\bE_{0}}^{2}+ \norm{\bF_{0}}^{2}}}=0.2462.
\end{equation*}
It is slightly larger than the relative distance for the single elasticity tensor case solved in \autoref{sec:num-cubic-elasticity}.

As $\bH^{**}$ and $\bK^{**}=k^{**} \bH^{**}$ are non zero, the two optimal tensors $\bE^{**}$ and $\bF^{**}$ are cubic (and so is the pair $(\bE^{**},\bF^{**})$). The relative distance of the given pair $(\bE_{0}, \bF_0)$ to isotropy,
\begin{equation*}
  \sqrt{\frac{\norm{\bE_{0}-\left(2\mu_{0} \bI +\lambda_{0} \Id \otimes \Id\right)}^{2}+ \norm{\bF_{0}-\left(2m_{0} \bI +\ell_{0} \Id \otimes \Id\right)}^{2}}
    {\norm{\bE_{0}}^{2}+ \norm{\bF_{0}}^{2}}
  }
  =0.5096,
\end{equation*}
is larger than the one to cubic symmetry.

The normal forms \eqref{eq:VoigtCubic} for both the optimal Hooke and Hill tensors
are finally obtained thanks to the computation of their invariants $\lambda$, $\mu$ $J_{2}$, $J_{3}$ and $\delta$, here evaluated first for $\bE^{**}$ and then for $\bF^{**}$ (by remark \ref{rem:normal-form-hooke}). Using \eqref{eq:CubicDirect} for each tensor $\bE^{**}$ and $\bF^{**}$, we get:
\begin{align*}
  [\bE^{**}_{\octa}] & ={\footnotesize
  \begin{pmatrix}
    229.484303
      &
    140.424515
      & 140.424515
      & 0          & 0          & 0                                    \\
    140.424515
      & 229.484303
      & 140.424515
      & 0          & 0          & 0                                    \\
    140.424515
      & 140.424515 & 229.484303
      & 0          & 0          & 0                                    \\
    0 & 0          & 0          & 131.757849
      & 0          & 0                                                 \\
    0 & 0          & 0          & 0          & 131.757849
      & 0                                                              \\
    0 & 0          & 0          & 0          & 0          & 131.757849
  \end{pmatrix}_{(\ee_{1}, \ee_{2}, \ee_{3})}}
  \ \text{GPa},
  \\
  [\bF^{**}_{\octa}] & ={\footnotesize
  \begin{pmatrix}
    331.183705
      &
    -139.258519
      & -139.258519
      & 0           & 0          & 0                                 \\
    -139.258519
      & 331.183705
      & -139.258519
      & 0           & 0          & 0                                 \\
    -139.258519
      & -139.258519 & 331.183705
      & 0           & 0          & 0                                 \\
    0 & 0           & 0          & 49.074814
      & 0           & 0                                              \\
    0 & 0           & 0          & 0         & 49.074814
      & 0                                                            \\
    0 & 0           & 0          & 0         & 0         & 49.074814
  \end{pmatrix}_{(\ee_{1}, \ee_{2}, \ee_{3})}}
  \ \text{GPa},
\end{align*}
which can be approximated as
\begin{align*}
  [\bE^{**}_{\octa}] & ={ 
  \begin{pmatrix}
    229.5
      &
    140.5
      & 140.5
      & 0     & 0     & 0               \\
    140.5
      & 229.5
      & 140.5
      & 0     & 0     & 0               \\
    140.5
      & 140.5 & 229.5
      & 0     & 0     & 0               \\
    0 & 0     & 0     & 132
      & 0     & 0                       \\
    0 & 0     & 0     & 0   & 132
      & 0                               \\
    0 & 0     & 0     & 0   & 0   & 132
  \end{pmatrix}_{(\ee_{1}, \ee_{2}, \ee_{3})}}
  \text{GPa},
  \\
  [\bF^{**}_{\octa}] & ={ 
  \begin{pmatrix}
    331
      &
    -139
      & -139
      & 0    & 0   & 0            \\
    -139
      & 331
      & -139
      & 0    & 0   & 0            \\
    -139
      & -139 & 331
      & 0    & 0   & 0            \\
    0 & 0    & 0   & 49
      & 0    & 0                  \\
    0 & 0    & 0   & 0  & 49
      & 0                         \\
    0 & 0    & 0   & 0  & 0  & 49
  \end{pmatrix}_{(\ee_{1}, \ee_{2}, \ee_{3})}}
  \text{GPa}.
\end{align*}
These two normal forms are obtained in the same cubic basis $(\ee_{i})$. Finally, by \eqref{eq:C-L}, the Hill parameter associated with $\bF^{**}_{\octa}$ is
\begin{equation*}
  L=0.312949\approx 0.31.
\end{equation*}

\section{Recovering the natural basis of a cubic fourth-order constitutive tensor}
\label{sec:recovering-cubic-normal-form}

A continuum mechanics anisotropic constitutive law, such as elasticity, is not represented by a unique constitutive tensor $\bE$ but by the set of all elasticity tensors $Q\star\bE$ related to $\bE$ by a rotation $Q$. Mathematically speaking, the anisotropic material property is represented by the orbit
\begin{equation*}
  \Orb(\bE)=\set{Q\star \bE, \; \det Q =1}.
\end{equation*}
For a given cubic elasticity tensor $\bE$, there exists a tensor $\bE_\octa$ in its orbit that is fixed by all the transformations
of the orientation preserving octahedral group $\octa$. The tensor $\bE_\octa$ is the so-called \emph{normal form} of $\bE$, and has
\eqref{eq:VoigtCubic} as Voigt representation.

When a cubic constitutive tensor -- such as the tensors $\bE^{*}$, $\bE^{**}$ and $\bF^{**}$ of previous numerical applications sections -- is not expressed in its natural (cubic) basis, one needs
\begin{enumerate}
  \item to compute its normal form,
  \item and to compute the rotation $Q$ that puts it in its normal form.
\end{enumerate}
Task $(1)$ can be done in a straightforward manner, using \emph{Invariant Theory} (see remark \ref{rem:normal-form-hooke}).
Note that the polynomial ($\lambda$, $\mu$, $J_{2}$ and $J_{3}$) and rational ($\delta$) invariants then involved are computed in the working basis (in which are expressed $\bE^{*}$, $\bE^{**}$ and $\bF^{**}$) by explicit formulas, whereas the methodology proposed in \cite{SMB2020} needs the computation of the eigenvalues of the Kelvin $6\times 6$ matrix representation of the considered elasticity tensor.

In practice, there are several ways to perform task $(2)$: using \emph{Maxwell multipoles}~\cite{Bae1993} and solving a degree-8 polynomial equation in one variable, or solving the linear system~\cite[Appendix B]{ADD2020}
\begin{equation*}
  \bL(\ba):=\tr ( \bH \times \ba) = 0, \qquad \ba \in \HH^{2},
\end{equation*}
where $\bH$ is the fourth-order harmonic component of the considered (cubic) elasticity tensor $\bE$ (it will next be either $\bH^{*}$ or $\bH^{**}$ or $\bK^{**}$). Here, the product $\times$ is the generalized cross product between totally symmetric tensors, defined by \eqref{eq:SxT}, and the totally symmetric fifth-order tensor $\bH \times \ba$ has components
\begin{equation*}
  \left(\bH \times \ba\right)_{ijklm} = \left(a_{i r}\varepsilon_{rjs}H_{sklm}\right)^{s}.
\end{equation*}
Generically, the deviatoric tensor $\ba$, solution of the equation $\bL(\ba)=0$, is orthotropic and carries the cubic basis $(\ee_{i})$. We shall apply the second methodology, which reduces to solve the linear equation $\bL(\ba)=0$, once the components of a cubic elasticity tensor are given (in an arbitrarily oriented basis).

\begin{rem}
  To avoid useless computations, it is important to note that, given a cubic elasticity tensor $\bE=(\lambda, \mu, 0, 0, \bH)$, it is equivalent to solve
  \begin{equation*}
    \tr(\bH \times \ba)=0, \qquad \ba \in \HH^{2},
  \end{equation*}
  or to solve
  \begin{equation*}
    \tr(\bE^s \times \ba)=0, \qquad \ba \in \HH^{2},
  \end{equation*}
  where $\bE^s$ is the totally symmetric part of $\bE$.
\end{rem}

The leading harmonic part $\bH$ of $\bE$ is assumed to be known. Indeed, it has been computed in the previous applications sections for the three optimal tensors $\bE^{*}$, $\bE^{**}$ and $\bF^{**}$. The methodology to determine the rotation matrix $Q$ is the following.
\begin{enumerate}
  \item Compute a basis $(\ba_{1}, \ba_{2})$ of the two-dimensional space of solutions of the linear system $\bL(\ba)=0$.

  \item The pair of second-order tensors $(\ba_{1}, \ba_{2})$ is orthotropic~\cite{OKDD2021}. Hence, a random tensor $\ba = t\ba_{1} + s\ba_{2}$ in this subspace will be generically orthotropic (as also, almost certainly, both $\ba_{1}$, $\ba_{2}$ computed by a Computer Algebra System). For such an orthotropic tensor, an orthogonal basis of eigenvectors $\uu_{i}$ will provide the solution as the rotation matrix $Q=(\uu_{1}, \uu_{2},\uu_{3})$.
  \item The normal form $\bE_{\octa}$ of $\bE$ is then obtained as
        \begin{equation*}
          \bE_{\octa}=Q\star \bE,
          \qquad
          (\bE_{\octa})_{ijkl}=Q_{ip}Q_{jq}Q_{kr}Q_{ls} E_{pqrs}.
        \end{equation*}
\end{enumerate}

\subsection*{Rotation associated with the cubic normal form for $\bE^*$}

Let us first apply this methodology to the cubic tensor $\bE=\bE^*=(\lambda_{0}, \mu_{0}, 0, 0, \bH=\bH^*)$ the nearest to $\bE_{0}$ (given by \eqref{eq:Ecubic}).
A basis for the space of traceless solutions for the system $\bL(\ba)=\left.\tr(\bH^* \times \ba)=0\right.$ is
\begin{equation*}
  \ba_{1}={\footnotesize
  \begin{pmatrix}
    1             & 0.1719141925  & -0.4213723048 \\
    0.1719141925  & -0.3670248470 & 0             \\
    -0.4213723048 & 0             & -0.6329751530
  \end{pmatrix}
  },
\end{equation*}
\begin{equation*}
  \ba_{2}={\footnotesize
  \begin{pmatrix}
    0             & -0.1797028254 & -1.106498932 \\
    -0.1797028254 & 3.880108140   & 1            \\
    -1.106498932  & 1             & -3.880108140
  \end{pmatrix}
  },
\end{equation*}
and the associated eigenvectors matrix $Q$ is
\begin{equation}\label{eq:Q}
  Q={\footnotesize
  \begin{pmatrix}
    -0.244381988215706 & 0.112674376458757 & -0.963110548548488 \\ -0.966260413618850 & -0.111624674249115 & 0.232122263412857\\ 0.0813526699553721 & -0.987342097243103 & -0.136151849428213
  \end{pmatrix} .
  }
\end{equation}
Finally, the normal form of $\bE^*$ is $\bE^*_{\octa}=Q\star \bE^*$, and one recovers \eqref{eq:EcubicNormal}.

\subsection*{Rotation associated with the cubic normal forms for $\bE^{**}$ and $\bF^{**}$}

The methodology also applies  to the optimal cubic tensors $\bE^{**}$ and $\bF^{**}$ of \autoref{sec:num-cubic-elasto-plasticity} .
A basis for the space of traceless solutions for the system $\left.\tr(\bH^{**} \times \ba)=0\right.$ is
\begin{equation*}
  \ba_{1}={\footnotesize
  \left(
  \begin{array}{ccc}
      1            & 0           & -25.46740405 \\
      0            & 37.84482400 & 20.16410890  \\
      -25.46740405 & 20.16410890 & -38.84482400 \\
    \end{array}
  \right) , }
\end{equation*}
\begin{equation*}
  \ba_{2}={\footnotesize
  \left(
  \begin{array}{ccc}
      0           & 1            & 49.43123622  \\
      1           & -76.54548515 & -40.58449192 \\
      49.43123622 & -40.58449192 & 76.54548515  \\
    \end{array}
  \right)  ,}
\end{equation*}
and the associated eigenvectors matrix $Q$ is
\begin{equation}\label{eq:QEF}
  Q={\footnotesize
  \left(
  \begin{array}{ccc}
      -0.407028622361688 & 0.194165806279755  & -0.892539825582074 \\
      0.167497955759488  & -0.944709963469717 & -0.281899839903180 \\
      -0.897926575725844 & -0.264239899698699 & 0.352001619332188  \\
    \end{array}
  \right)
  }.
\end{equation}
The normal forms of $\bE^{**}$ and $\bF^{**}$ are $\bE^{**}_{\octa}=Q\star \bE^{**}$ and $\bF^{**}_{\octa}=Q\star \bF^{**}$. They  are (simultaneously) obtained for the (same) rotation $Q$. One then recovers the normal forms given at the end of \autoref{sec:num-cubic-elasto-plasticity}
(\emph{i.e.}, in Voigt notation, the matrices $[\bE^{**}_{\octa}]$ and $[\bF^{**}_{\octa}]$).

\section{Conclusion}

Thanks to the recent characterization of the cubic elasticity symmetry classes by polynomial covariants \cite{OKDD2021}, we have formulated the \emph{distance to cubic symmetry problem} as a polynomial optimization problem, and derived the associated Euler--Lagrange equations. We have used the theory of Gröbner bases to solve these equations, in a quasi-analytical manner (using a Computer Algebra System). This methodology has been applied to the case of a single elasticity tensor, as well as to the case of a pair of Hooke and Hill elasto-plasticity tensors. Besides, we have recovered the normal forms of the optimal cubic elasticity/plasticity tensors.

The key-point of the study is that the corresponding cubic symmetry is defined by a polynomial tensorial equation, which is a submersion (apart from the isotropic singularity, which is controlled). This makes it possible to apply the Euler--Lagrange method and use Gröbner bases to compute the critical points.

\appendix

\section{Solving algebraic systems using Gröbner bases}
\label{sec:groebner}

In this appendix, we propose to explain how to use \emph{Gröbner bases} to solve non-linear algebraic systems. Our goal is not to summarize the theory of Gröbner bases, nor to introduce the basics of algebraic geometry but to explain through some examples how it works. For more details on this topic and a deeper insight, we strongly recommend the following books~\cite{CLO2007,Stu1993}, which
contain a lot of references.

Gröbner bases were introduced in the sixties by Buchberger \cite{Buc1965}. Like Gaussian elimination method is used to solve a system of linear equations, Gröbner bases are useful to solve a system of non-linear algebraic equations
\begin{equation}\label{eq:algebraic-equations}
  \left\{
  \begin{split}
    & f_{1}(x_{1}, \dotsc , x_{n}) & = 0 \\
    & \dotsb & \\
    & f_{m}(x_{1}, \dotsc , x_{n}) & = 0
  \end{split}
  \right.
\end{equation}
where $f_{1}, \dotsc , f_{m}$ are polynomial functions in the variables $x_{1}, \dotsc , x_{n}$. Note however that in general, and even for one variable, it is useless to search for closed-form solutions. Therefore, what is expected is a procedure which produces a new system of algebraic equations which is \emph{simpler}. Contrary to Gauss elimination algorithm, where the variables are naturally ordered by the choice of a basis, we need to choose a total order on monomials in order to make the Gröbner bases algorithm to work. There are many total orders on monomials in several variables, the most common being the lexicographic order induced by $x_{1} < x_{2} < \dotsb < x_{n}$, and the resulting Gröbner basis will depend drastically on the choice of an order.

Let us illustrate what we mean here through an example. Consider, for instance, an intersection of three quadrics in $\RR^{3}$, given by the following non-linear system of three homogeneous polynomial equations of degree 2
\begin{equation}\label{eq:pol-system-1}
  \left\{
  \begin{split}
    x_{3}^{2} + x_{2}^{2} + x_{1}^{2} & = 1 \\
    x_{3}^{2} + x_{1}x_{2} & = 1 \\
    x_{1}x_{3} + x_{1}x_{2} & = 2.
  \end{split}
  \right.
\end{equation}
in the three variables $(x_{1}, x_{2},x_{3})$. The computation of a Gröbner basis for this system (with the lexicographic order induced by $x_{1} < x_{2} < x_{3}$) leads to the following equivalent system of equations
\begin{equation*}
  \left\{
  \begin{split}
    & 3 {x_{3}^{8}} + 3 {x_{3}^{6}} + 6 {x_{3}^{4}} + 3 {x_{3}^{2}} + 1  = 0 \\
    & x_{2} = - \frac{3}{2} {x_{3}^{7}} - 3 {x_{3}^{3}} + \frac{1}{2}x_{3} \\
    & x_{1} = -3 {x_{3}^{7}}-3 {x_{3}^{5}}-6 {x_{3}^{3}}-2 x_{3}.
  \end{split}
  \right.
\end{equation*}
Therefore, in this example, computing a Gröbner basis for a system of $n$ equations in $n$ variables leads to an equivalent system of equations of the form
\begin{equation}\label{eq:GB-system-1}
  \left\{
  \begin{split}
    & P_{n}(x_{n}) = 0 \\
    & x_{n-1} = P_{n-1}(x_{n}) \\
    & \dotsb  \\
    & x_{1} = P_{1}(x_{n})
  \end{split}
  \right.
\end{equation}
In other words, in that case, one has been able to reduce the non-linear algebraic system \eqref{eq:pol-system-1} to an equivalent triangular system consisting in one polynomial equation in the last variable $x_{n}$ and a list of $n-1$ equations which are solved in the remaining variables $x_{n-1}, \dotsc , x_{1}$. In particular, such a system \emph{has at most a finite number of solutions}.

This is generally what happens if one tries to solve an algebraic system of $n$ equations in $n$ variables but there exists, nevertheless, some \emph{degenerate} situations (as it is the case for linear systems when the determinant of the system vanishes). In the next example, we will illustrate this degeneracy. Consider the following set of equations
\begin{equation}\label{eq:pol-system-2}
  \left\{
  \begin{split}
    x_{1}^{2} + x_{2}^{2} + x_{3}^{2} & = 1 \\
    x_{1}x_{2} + x_{3}^{2} & = 1 \\
    x_{2}^{2} - x_{1} x_{2} + x_{1}^{2} & = 0.
  \end{split}
  \right.
\end{equation}
An equivalent system given by the computation of a Gröbner basis (for the lexicographic order induced by $x_{1} < x_{2} < x_{3}$) is
\begin{equation*}
  \left\{
  \begin{split}
    {{x}_{3}^{2}}+{x_{1}} {x_{2}}-1 & = 0 \\
    -{x_{2}} {{x}_{3}^{2}}+{x_{1}} {{x}_{3}^{2}}-{{x}_{2}^{3}}+{x_{2}}-{x_{1}} & = 0 \\
    {{x}_{3}^{4}}+{{x}_{2}^{2}}\, {{x}_{3}^{2}}-2 {{x}_{3}^{2}}+{{x}_{2}^{4}}-{{x}_{2}^{2}}+1 & = 0 \\
    {{x}_{3}^{2}}+{{x}_{2}^{2}}+{{x}_{1}^{2}}-1 & = 0
  \end{split}
  \right.
\end{equation*}
This time, the explicit solution of the problem is far less straightforward (it will be given and explained anyway below). What we can observe, however, is that the third equation in~\eqref{eq:pol-system-2} is a linear combination of the first two ones. Thus, the system is in fact rectangular (two equations in three variables, rather than three equations). More generally, this situation appears for systems of $n$ algebraic equations with $n$ variables each time the polynomials $f_{1}, \dotsc , f_{n}$ are \emph{algebraically dependent}.

Let us discuss now in which way we can interpret this procedure as an extension to systems of polynomial equations of the Gaussian elimination algorithm. In Gauss algorithm, a succession of invertible linear transformations reduces a general system of linear equations into one which is triangular. The first equation involves all the variables, the second equation does not involve $x_{1}$, the third equation does not involved $x_{1}, x_{2}$, \ldots. In the nonlinear case a similar process occurs somehow but requires, to be described correctly, to define the notion of \emph{ideal}.

An ideal $I$ of the algebra $\CC[x_{1}, \dotsc , x_{n}]$ is a subalgebra of $\CC[x_{1}, \dotsc , x_{n}]$ which is stable by multiplication by every polynomial in $\CC[x_{1}, \dotsc , x_{n}]$. More precisely, this means that
\begin{equation*}
  f \in I \quad \text{and} \quad p \in \CC[x_{1}, \dotsc , x_{n}] \implies pf \in I.
\end{equation*}
Now consider the system~\eqref{eq:algebraic-equations} and the ideal
\begin{equation*}
  I := \set{p_{1}f_{1} + \dotsb + p_{m}f_{m};\; p_{k} \in \CC[x_{1}, \dotsc , x_{n}]},
\end{equation*}
generated by $f_{1}, \dotsc , f_{m}$. It is clear that every $f \in I$ vanishes on each solution of~\eqref{eq:algebraic-equations}, and conversely that the solutions of~\eqref{eq:algebraic-equations} can be recast as the solutions of the infinite system of equations
\begin{equation*}
  f(x_{1}, \dotsc , x_{n}) = 0, \qquad f \in I.
\end{equation*}
Now, let us introduce the $l$-th \emph{elimination ideal}
\begin{equation*}
  I_{l} := I \cap \CC[x_{l+1}, \dotsc , x_{n}], \qquad l = 1, \dotsc , n-1.
\end{equation*}
Note that $I_{l}$ is an ideal of $\CC[x_{l+1}, \dotsc , x_{n}]$ but not of $\CC[x_{1}, \dotsc , x_{n}]$. It is however a subalgebra of $\CC[x_{1}, \dotsc , x_{n}]$ and we have
\begin{equation*}
  I_{n-1} \subset \dotsb \subset I_{1} \subset I.
\end{equation*}
Then, a Gröbner basis (computed using the lexicographic order induced by $x_{1} < x_{2} < \dotsb < x_{n}$) provides a new system of generators of $I$ which is compatible with the sequence of elimination ideals. Let us illustrate what we mean here, using our first example~\eqref{eq:pol-system-1}. In that case, the following Gröbner basis was computed
\begin{equation*}
  \left\{
  \begin{split}
    g_{1} & = -3 {{x}_{3}^{7}}-3 {{x}_{3}^{5}}-6 {{x}_{3}^{3}}-2 {{x_{3}}-{x_{1}}}, \\
    g_{2} & = 3 {{x}_{3}^{7}}+6 {{x}_{3}^{3}}-{x_{3}}+2 {x_{2}}, \\
    g_{3} & = -3 {{x}_{3}^{8}}-3 {{x}_{3}^{6}}-6 {{x}_{3}^{4}}-3 {{x}_{3}^{2}}-1.
  \end{split}
  \right.
\end{equation*}
In this example, $I$ is generated by $g_{1}, g_{2}, g_{3}$, $I_{1}$ by $g_{2}, g_{3}$ and $I_{2}$ by $g_{3}$. It is in this sense that a Gröbner basis can be considered as a \emph{triangulation} of the initial problem. Hence, in this example, solving the problem consists first in finding the roots of $g_{3}=0$, then calculating $x_{2}$ using $g_{2}=0$ and then $x_{1}$ using $g_{1}=0$. Consider now the second example~\eqref{eq:pol-system-2}. In that case, the following Gröbner basis was computed using the lexicographic order $x_{1} < x_{2} < x_{3}$
\begin{equation*}
  \left\{
  \begin{split}
    g_{1} & = {{x}_{3}^{2}}+{x_{1}} {x_{2}}-1, \\
    g_{2} & = -{x_{2}} {{x}_{3}^{2}}+{x_{1}} {{x}_{3}^{2}}-{{x}_{2}^{3}}+{x_{2}}-{x_{1}}, \\
    g_{3} & = {{x}_{3}^{2}}+{{x}_{2}^{2}}+{{x}_{1}^{2}}-1, \\
    g_{4} & = {{x}_{3}^{4}}+{{x}_{2}^{2}}\, {{x}_{3}^{2}}-2 {{x}_{3}^{2}}+{{x}_{2}^{4}}-{{x}_{2}^{2}}+1.
  \end{split}
  \right.
\end{equation*}
In this example, $I$ is generated by $g_{1}, g_{2}, g_{3}, g_{4}$, $I_{1}$ by $g_{4}$ and $I_{2}$ by $0$. We could continue here to explain the complete resolution of the problem but it appears that changing the order on monomials makes the resolution by far much more readable for a human being. Indeed, changing the lexicographic order $x_{1} < x_{2} < x_{3}$ to $x_{3} < x_{1} < x_{2}$ leads to the following Gröbner basis
\begin{equation*}
  \left\{
  \begin{split}
    g_{1} & = x_{3}^{2} - 1 + x_{1} x_{2}, \\
    g_{2} & = x_{1}^{2}-x_{1} x_{2} + x_{2}^{2},
  \end{split}
  \right.
\end{equation*}
and, then, $I$ is generated by $g_{1}, g_{2}$, $I_{1}:= I \cap \CC[x_{1},x_{2}]$ by $g_{2}$ and $I_{2} := I \cap \CC[x_{2}]$ by $0$. We will now proceed to the complete resolution of the system. First, we need to solve the equation in one variable $x_{2}$ given by $I_{2}$. Since $I_{2}$ is generated by $0$, this means that the variable $x_{2}$ is free. We will thus set $x_{2}=t$ ($t \in \CC$). Then, we need to solve the system of equations in two variables $(x_{2},x_{1})$ given by $I_{1}$. More precisely, since we have already solved the problem for $x_{2}$ (the system is triangular), we seek solutions $(x_{2},x_{1})$ of
\begin{equation*}
  x_{1}^{2}-x_{1} x_{2} + x_{2}^{2} = 0,
\end{equation*}
which \emph{extend} the solution $x_{2} = t$. Hence, $x_{2}$ is no more a variable here but a parameter. This equation has either one solution $(x_{2}=0,x_{1}=0)$ if $t=0$ or two conjugate imaginary solutions if $t \ne 0$. Now, we need to solve the system of equations in three variables $(x_{2},x_{1},x_{3})$ given by $I$. Hence, we need to solve the equation
\begin{equation*}
  x_{3}^{2} - 1 + x_{1} x_{2} = 0
\end{equation*}
but where $(x_{2}=a,x_{1}=b)$ is a solution of the previous step and where $x_{2}=a$ and $x_{1}=b$ should be considered as parameters of the problem. Such a solution is said \emph{to extend} the previous one. In our example, we find exactly two solutions for $x_{3}$ for each solution $(x_{2}=a,x_{1}=b)$ (because the coefficient of $x_{3}^{2}$ is one). In other examples, nevertheless, some solutions could be  \emph{not extendable} (for example, if the coefficient of $x_{3}^{2}$ depends on $a$, $b$ and vanishes for some values of $a$ and $b$). This example illustrate the \emph{triangular process} allowed by the computation of a Gröbner basis in solving non-linear algebraic equations. Note finally that solutions are sought in $\CC$. In this example, there are an infinite number of complex solutions. Of course, it may happen that there are no real solution at all.

\begin{rem}\label{rem:coeffQ}
  The Gröbner bases are exact when computed over the field $\QQ$ of rational numbers. It is a natural question whether or not, one could work with Groebner bases with coefficients in the field of real or complex numbers, or to be more exact, using floating numbers. In practice, this is a difficult topic since there are convergence/accuracy issues. Anyway, this subject is a research area called \emph{Groebner bases with coefficients in an inexact field}. We redirect the interested reader to~\cite{Nag2009}.
\end{rem}


\end{document}